\documentclass[11pt]{article}
\usepackage[margin=1in]{geometry}

\usepackage{amsmath,amsthm,amsfonts,amssymb}
\usepackage{mathrsfs}
\usepackage{mathtools}
\usepackage{graphicx,enumerate}
\graphicspath{ {./figures/} }
\usepackage[title]{appendix}
\usepackage[
            CJKbookmarks=true,
            bookmarksnumbered=true,
            bookmarksopen=true,
            linktocpage=true,
            colorlinks=true,
            citecolor=red,
            linkcolor=blue,
            anchorcolor=red,
            urlcolor=blue
            ]{hyperref}
\usepackage[font=small,labelfont=bf]{caption} 
\usepackage{stmaryrd}

\newcommand{\eps}{\varepsilon}
\renewcommand{\P}{\mathbb{P}}

\usepackage{tikz}
\usepackage{tikz-qtree}

\usepackage{tikz-cd}

\usetikzlibrary{calc,arrows,cd}
\usetikzlibrary{decorations.pathreplacing,decorations.markings}

\tikzstyle{dot}=[circle,fill,black,inner sep=1pt]

\tikzset{
  on each segment/.style={
    decorate,
    decoration={
      show path construction,
      moveto code={},
      lineto code={
        \path [#1]
        (\tikzinputsegmentfirst) -- (\tikzinputsegmentlast);
      },
      curveto code={
        \path [#1] (\tikzinputsegmentfirst)
        .. controls
        (\tikzinputsegmentsupporta) and (\tikzinputsegmentsupportb)
        ..
        (\tikzinputsegmentlast);
      },
      closepath code={
        \path [#1]
        (\tikzinputsegmentfirst) -- (\tikzinputsegmentlast);
      },
    },
  },
  mid arrow/.style={postaction={decorate,decoration={
        markings,
        mark=at position .5 with {\arrow[#1]{stealth}}
      }}},
  early arrow/.style={postaction={decorate,decoration={
        markings,
        mark=at position .2 with {\arrow[#1]{stealth}}
      }}},
}

\tikzstyle{int}=[draw, fill=blue!15, minimum size=2em]
\tikzstyle{init} = [pin edge={to-,thin,black}]

\def\alternatecolorred{%
    \pgfkeysalso{red}%
    \global\let\alternatecolor\alternatecolorblue 
}
\def\alternatecolorblue{%
    \pgfkeysalso{blue}%
    \global\let\alternatecolor\alternatecolorred 
}

\newcommand{\altred}{\let\alternatecolor\alternatecolorred 
\tikzset{every edge/.append code = {%
    \global\let\currenttarget\tikztotarget 
    \pgfkeysalso{append after command={(\currenttarget)}}
			\alternatecolor
}}
}
\newcommand{\altblue}{\let\alternatecolor\alternatecolorblue 
\tikzset{every edge/.append code = {%
    \global\let\currenttarget\tikztotarget 
    \pgfkeysalso{append after command={(\currenttarget)}}
			\alternatecolor
}}
}

\tikzstyle{vertexdot}=[circle, draw, fill=black, minimum size=3,inner sep=0pt]

\newtheorem{theorem}{Theorem}
\newtheorem{lemma}{Lemma}[section]
\newtheorem{proposition}{Proposition}[section]
\newtheorem{corollary}{Corollary}

\theoremstyle{definition}

\newtheorem{definition}{Definition}

\usepackage{color}
\usepackage{subfigure}
\usepackage{algorithm}
\usepackage{algorithmic}

\usepackage{xspace,prettyref}

\newrefformat{eq}{(\ref{#1})}
\newrefformat{chap}{Chapter~\ref{#1}}
\newrefformat{sec}{Section~\ref{#1}}
\newrefformat{alg}{Algorithm~\ref{#1}}
\newrefformat{fig}{Fig.~\ref{#1}}
\newrefformat{tab}{Table~\ref{#1}}
\newrefformat{rmk}{Remark~\ref{#1}}
\newrefformat{clm}{Claim~\ref{#1}}
\newrefformat{def}{Definition~\ref{#1}}
\newrefformat{cor}{Corollary~\ref{#1}}
\newrefformat{lmm}{Lemma~\ref{#1}}
\newrefformat{prop}{Proposition~\ref{#1}}
\newrefformat{app}{Appendix~\ref{#1}}
\newrefformat{hyp}{Hypothesis~\ref{#1}}
\newrefformat{thm}{Theorem~\ref{#1}}
\newrefformat{ass}{Assumption~\ref{#1}}
\newrefformat{conj}{Conjecture~\ref{#1}}

\newcommand{\pth}[1]{\left( #1 \right)}
\newcommand{\qth}[1]{\left[ #1 \right]}
\newcommand{\sth}[1]{\left\{ #1 \right\}}

\newcommand{\tH}{{\widetilde{H}}}

\newcommand{\tU}{{\widetilde{U}}}
\newcommand{\tV}{{\widetilde{V}}}
\newcommand{\tW}{{\widetilde{W}}}

\newcommand{\sfA}{{\mathsf{A}}}

\newcommand{\calA}{{\mathcal{A}}}

\newcommand{\calE}{{\mathcal{E}}}
\newcommand{\calF}{{\mathcal{F}}}

\newcommand{\calH}{{\mathcal{H}}}

\newcommand{\calN}{{\mathcal{N}}}

\newcommand{\calP}{{\mathcal{P}}}

\newcommand{\calT}{{\mathcal{T}}}

\newcommand{\calX}{{\mathcal{X}}}
\newcommand{\calY}{{\mathcal{Y}}}

\newcommand{\frakS}{{\mathfrak{S}}}

\DeclareMathAlphabet{\varmathbb}{U}{bbold}{m}{n}

\renewcommand{\d}{{\rm d}}

\renewcommand{\hat}{\widehat}
\renewcommand{\tilde}{\widetilde}

\usepackage{bbm}

\newcommand{\R}{\mathbb{R}}
\newcommand{\C}{\mathbb{C}}
\newcommand{\N}{\mathbb{N}}
\newcommand{\Z}{\mathbb{Z}}

\newcommand{\1}{\mathbbm{1}}

\DeclareMathOperator{\Tr}{Tr}

\newcommand{\abs}[1]{\left| {#1} \right|}

\newcommand{\sT}{{\mathscr{T}}}

\newcommand{\E}{\mathbb{E}}

\renewcommand{\S}{\mathfrak{S}}
\renewcommand{\P}{\mathbb{P}}

\newcommand{\A}{\mathbb A}

\DeclarePairedDelimiterX{\infdivx}[2]{\Big(}{\Big)}{%
  #1\;\delimsize\|\;#2%
}

\renewcommand{\geq}{\geqslant}
\renewcommand{\leq}{\leqslant}
\newcommand{\ep}{\epsilon}

\newcommand{\im}{{\mathrm{Im} \,}}
\newcommand{\re}{{\mathrm{Re} \,}}

\newcommand{\MP}{\mathsf{MP}}
\renewcommand{\sc}{\mathsf{sc}}

\renewcommand{\i}{\text{\rm i}}

\renewcommand{\d}{\mathrm{d}}

\newcommand{\uu}{\hat{\varphi}}

\DeclareMathOperator{\dist}{dist}
\newcommand{\La}{\Lambda}

\newcommand{\Us}{\sT_{\mathrm{short}}}

\renewcommand{\A}{\mathscr{A}}
\renewcommand{\S}{\mathscr{S}}
\newcommand{\Id}{\mathsf{Id}}
\newcommand{\de}{\delta}

\renewcommand{\flat}{{\calF}}
\newcommand{\av}{\calA}

\renewcommand{\d}{\mathrm{d}}
\newcommand{\m}{\Gamma}

\begin{document}

\pgfdeclarelayer{background}
\pgfdeclarelayer{foreground}
\pgfsetlayers{background,main,foreground}

\title{
Optimal Smoothed Analysis and Quantitative Universality for the Smallest Singular Value of Random Matrices
}

\author{
Haoyu Wang \thanks{Department of Mathematics, Yale University, New Haven, CT 06520, USA,  \texttt{haoyu.wang@yale.edu}}
}

\date{\today}

\maketitle

\begin{abstract}

The smallest singular value and condition number play important roles in numerical linear algebra and the analysis of algorithms. In numerical analysis with randomness, many previous works make Gaussian assumptions, which are not general enough to reflect the arbitrariness of the input. To overcome this drawback, we prove the first quantitative universality for the smallest singular value and condition number of random matrices. 

Moreover, motivated by the study of smoothed analysis that random perturbation makes deterministic matrices well-conditioned, we consider an analog for random matrices. For a random matrix perturbed by independent Gaussian noise, we show that this matrix quickly becomes approximately Gaussian. In particular, we derive an optimal smoothed analysis for random matrices in terms of a sharp  Gaussian approximation.
\end{abstract}





\newpage

\tableofcontents

\newpage

\section{Introduction}

In numerical analysis and analysis of algorithms, the smallest singular value and condition number of matrices play important roles \cite{trefethen1997numerical}. Broadly speaking, the smallest singular value reflects the invertibility of a matrix, and a better non-degeneracy condition of the matrix will result in faster algorithms or ensures better stability in computations. The condition number is a crucial quantity as well. It measures the hardness and accuracy of numerical computations, and its most well-known practical meaning is the loss of precision in solving linear equations \cite{smale1985efficiency}. For more concrete state-of-the-art applications, the smallest singular value affect the stability in fast algorithms of solving linear systems \cite{peng2021solving,nie2022matrix}. In \cite{ghadiri2021sparse}, the condition number determines the iteration complexity in fast algorithms for $ p $-norm regression. There are a lot of other applications of these two quantities, and we do not intend to provide a complete review here.

In the seminal work of Spielman and Teng \cite{spielman2004smoothed}, smoothed analysis is introduced to understand why some algorithms with poor worst-case performance can work successfully in practice. Roughly speaking, smoothed analysis is an interpolation between the worst-case analysis and the average-case analysis. In the context of solving linear equations $ Ax=b $, even if the matrix $ A $ has large condition number (and consequently large loss of precision), a random perturbation $ A+H $ will become well-conditioned with high probability. Smoothed analysis for random perturbation of deterministic matrices has been well studied since its invention \cite{wschebor2004smoothed,van2014probability,tao2010smooth,farrell2016smoothed,shah2020smoothed} etc. It has also been applied in many algorithms, for example Gaussian elimination \cite{sankar2006smoothed}, matrix inversion \cite{burgisser2010smoothed}, conjugate descent method \cite{menon2016smoothed}, tensor decomposition \cite{bhaskara2014smoothed}, the $ k $-means method \cite{arthur2011smoothed}, etc.

Classical results of smoothed analysis focus on random perturbation of deterministic matrices, and show that this perturbation indeed makes the original matrix better for the implementation of algorithms (in the sense such as having smaller condition number). Recently, numerical computations with random data has become more and more common. In the context of random matrices, we show that if a random matrix is perturbed by independent Gaussian noise, it quickly becomes approximately Gaussian. The smallest singular value and condition number for a Gaussian matrix has been well studied \cite{edelman1988}, and therefore such Gaussian approximations enable us to analyze the more tractable Gaussian ensemble instead with benign approximation error. Specifically, we prove an optimal smoothed analysis of random matrices in terms of a sharp Gaussian approximation.

Moreover, in numerical analysis with randomness, many previous works make Gaussian assumptions. Such assumptions are not general enough to reflect the arbitrariness of the input. To overcome this drawback, we prove the first quantitative universality for the smallest singular value and condition number of random matrices. Both the smoothed analysis for random matrices and the quantitative universality will play useful roles in computations with randomness.

\subsection{Overview and our contributions}

From the mathematical perspective, the edge universality has been a classical problem in random matrix theory and it has tremendous progress in the past decade. However, most previous results are qualitative statements, and most quantitative rate of convergence to the limiting law were only known for integrable models. Recently, the rate of convergence to the Tracy-Widom distribution for the largest eigenvalue was first obtained in \cite{bourgade2021extreme} for generalized Wigner matrices and in \cite{wang2019quantitative} for sample covariance matrices. These results were further improved by Schnelli and Xu in \cite{schnelli2022wigner,schnelli2021covariance,schnelli2022quantitative}. Our work is the first result about the quantitative universality for the smallest singular value and the condition number.

For an $ M \times N $ matrix $ H $ with $ \lim_{N \to \infty} N/M \to d \in (0,1] $, the empirical spectral distribution of the sample covariance matrix $ H^\top H $ converges to the Marchenko-Pastur law
$$ \rho_{\mathsf{MP}}(x) = \frac{1}{2\pi d} \sqrt{\frac{(x-(1-\sqrt{d})^2)((1+\sqrt{d})^2-x)}{x^2}},\ \ x \in [(1-\sqrt{d})^2,(1+\sqrt{d})^2]. $$
In the case $ d<1 $, the Marchenko-Pastur law has no singularity and such situations are called the soft edge case in the literature of random matrix theory. Moreover, note that all eigenvalues are of order constant with high probability. In particular, thanks to the square-root decay near the spectral edge of the Marchenko-Pastur law, the extreme eigenvalues of the covariance matrix $ H^\top H $ are highly concentrated at the spectral endpoints at the scale $ N^{-2/3} $ and the fluctuation is the Tracy-Widom law. In this case, the optimal rate of convergence to Tracy-Widom distribution has been established in \cite{schnelli2021covariance}, and the quantitative universality for the smallest singular value will be an easy consequence.

However, in the case $ \lim N/M =1 $, things become much more complicated. Note that in this case the Marchenko-Pastur law has a singularity at the origin
$$ \rho_{\MP}(x) = \dfrac{1}{2\pi}\sqrt{\dfrac{4-x}{x}},\ \ x \in [0,4]. $$
This singularity of the limiting spectral distribution makes the behaviours of the eigenvalue near the left endpoint $ x=0 $ very different. In particular, the typical eigenvalue spacing near the left edge is of order $ N^{-1} $, which is much smaller than the soft edge case where the typical spacing is of order $ N^{-2/3} $. In random matrix theory, this case is called the \emph{hard edge} case. In this model, the left edge of the spectrum with singularity is called the hard edge and the right edge is called the soft edge. The study of the spectral statistics near the hard edge is a notoriously tricky problem. Due to some technical difficulties, the study of the hard edge case is mostly restricted to square matrices with $ M \equiv N $. Because of this difficulty, we will also focus on the study of square $ N \times N $ random matrices.

For the study of the smallest singular value of an $ N \times N $ random matrix, the first result was obtained by Edelman for the Gaussian matrix in \cite{edelman1988}. Later, the distribution in the Gaussian ensemble was shown to be universal by Tao and Vu in \cite{tao2010random} and further generalized to sparse matrices by Che and Lopatto in \cite{che2019universality}. However, both of their universality results are qualitative statements with unknown error terms.

In computer science and numerical computations, the accuracy or complexity of algorithms depend on the smallest singular value or condition number in a delicate way. Therefore, a qualitative universality is not enough to measure the performance of general random input. Thus, quantitative estimates are necessary for practical purposes.
In this work, we prove the first quantitative version of the universality for the smallest singular value and the condition number. This solves a long-standing open problem\footnote{Problem 11 in \href{http://www.aimath.org/WWN/randommatrices/randommatrices.pdf}{http://www.aimath.org/WWN/randommatrices/randommatrices.pdf}, \textit{Open problems: AIM workshop on random matrices.}}. Along the way to prove quantitative universality, we obtain a sharp estimate for matrices with Gaussian perturbations, and hence establish the optimal smoothed analysis for random matrices.


\subsection{Models and main results}

Thanks to the motivations from theoretical computer science, we mainly focus on real matrices. However, as mentioned in \cite{che2019universality}, our whole proof works for complex matrices as well. 


Let $ H=(H_{ij}) $ be an $ N \times N $ matrix with independent real valued entries with mean 0 and variance $ N^{-1} $,
\begin{equation}\label{e.Assumption1}
H_{ij} = N^{-1/2}h_{ij},\ \ \ \E h_{ij}=0,\ \ \ \E h_{ij}^2 =1.
\end{equation}
We assume the entries $ h_{ij} $ have a sub-exponential decay, that is, there exists a constant $ \theta>0 $ such that for $ u>1 $,
\begin{equation}\label{e.Assumption2}
\P (|h_{ij}| > u) \leq \theta^{-1} \exp (-u^\theta).
\end{equation}
We remark that this assumption is mainly for convenience, and other conditions such as the existence of a sufficiently high moment would also be enough.

For an $ N \times N $ matrix $ H $, let $ \sigma_1(H) \leq \cdots \leq \sigma_N(H) $ denote the singular values in non-decreasing order and use $ \kappa(H) := \sigma_N(H)/\sigma_1(H) $ to denote the condition number. Throughout this paper, we let $ G $ be an $ N \times N $ Gaussian matrix with i.i.d. entires $ \calN(0,N^{-1}) $.

To state the main results, we first introduce two important probabilistic notions that are commonly used throughout the whole paper.

\begin{definition}[Overwhelming probability]
Let $ \{\calE_N\} $ be a sequence of events. We say that $ \calE_N $ holds with overwhelming probability if for any (large) $ D>0 $, there exists $ N_0(D) $ such that for all $ N \geq N_0(D) $ we have
$$ \P(\calE_N) \geq 1-N^{-D}. $$
\end{definition}

\begin{definition}[Stochastic domination]
Let $ \calX=\{\calX_N\} $ and $ \calY=\{\calY_N\} $ be two families of nonnegative random variables. We say that $ \calX $ is stochastically dominated by $ \calY $ if for all (small) $ \eps>0 $ and (large) $ D>0 $, there exists $ N_0(\eps,D)>0 $ such that for $ N \geq N_0(\eps,D) $ we have
$$ \P \pth{ \calX_N > N^{\eps} \calY_N } \leq N^{-D}. $$
The stochastic domination is always uniform in all parameters. If $ \calX $ is stochastically dominated by $ \calY $, we use the notation $ \calX \prec \calY $.
\end{definition}

Our main result is the following.

\begin{theorem}\label{thm:Smoothed_Singular}
Let $ H $ be an $ N \times N $ random matrix satisfying \eqref{e.Assumption1} and \eqref{e.Assumption2}. For any $ \eps>0 $ and any $ \lambda>N^{-1/2+\eps} $, we have
\begin{equation}\label{eq:Smoothed_Singular}
\abs{\sigma_1(H+\lambda G) - \sqrt{1+\lambda^2}\sigma_1(G)} \prec \frac{\sqrt{1+\lambda^2}}{N^2 \log (1+\lambda^2)}.    
\end{equation}
\end{theorem}

\begin{theorem}\label{t.Real}
Let $ H $ be an $ N \times N $ matrix satisfying \eqref{e.Assumption1} and \eqref{e.Assumption2}. For any $ \delta \in (0,1) $ and $ \eps>0 $, we have
\begin{multline}\label{e.Real}
\P \pth{ N\sigma_1(G) > r + N^{-\delta} } - N^\eps \pth{N^{-1+\delta} \vee N^{-\frac{1}{2}} } \leq \P \pth{N \sigma_1(H)>r}\\
\leq \P \pth{ N\sigma_1(G) > r - N^{-\delta} } + N^\eps \pth{N^{-1+\delta} \vee N^{-\frac{1}{2}} },
\end{multline}
where $ a \vee b = \max\{a,b\} $ denotes the maximum between $ a $ and $ b $.
\end{theorem}

For the smallest singular value, one aspect of the complex-valued case is particularly interesting in the sense that the complex Gaussian model is explicitly integrable, i.e., the distribution of its smallest singular value is given by an exact formula. Specifically, let $ G_\C $ be an $ N \times \N $ matrix whose entries are i.i.d complex Gaussians whose real and imaginary parts are i.i.d. copies of $\tfrac{1}{\sqrt{2}}\calN(0,N^{-1}) $.
For the complex Gaussian ensemble, Edelman proved in \cite{edelman1988} that the distribution of the (renormalized) smallest singular value of a complex Gaussian ensemble is independent of $ N $ and can be computed explicitly
$$ \P(N \sigma_1(G_\C) \leq r) =\int_0^r e^{-x} \d x=1-e^{-r}. $$
Thanks to this exact formula for the integrable model, the edge universality for the smallest singular value can be quantified in terms of the Kolmogorov-Smirnov distance to the explicit law.

More precisely, let $ H_\C $ be an $ N \times N $ random matrix satisfying
$$ (H_\C)_{ij} = N^{-1/2}h_{ij},\ \ \ \E h_{ij}=0,\ \ \ \E [(\re h_{ij})^2]=\E [(\im h_{ij})^2] =\frac{1}{2},\ \ \E[(\re h_{ij})( \im h_{ij})]=0.   $$
and the sub-exponential decay assumption \eqref{e.Assumption2}. Then we have the following rate of convergence to the explicit law.

\begin{corollary}\label{c.Complex}
Let $ H_\C $ be a complex $ N \times N $ random matrix defined as above, then for any $ \eps>0 $ we have
\begin{equation}\label{e.Complex}
\P \pth{N \sigma_1(H_\C)  \leq r} = 1-e^{-r} + O(N^{-\frac{1}{2}+\eps}).
\end{equation}
\end{corollary}



\smallskip

Based on the analysis of the smallest singular value, combined with results on the quantitative results of the largest singular values, we can also derive optimal smoothed analysis and quantitative universality for the condition number.

\begin{theorem}\label{thm:Smoothed_Condition}
Let $ H $ be an $ N \times N $ random matrix satisfying \eqref{e.Assumption1} and \eqref{e.Assumption2}.  For any $ \eps>0 $ and any $ \lambda>N^{-1/2+\eps} $, we have
\begin{equation}\label{eq:Smoothed_Condition}
\abs{\kappa(H+\lambda G) - \kappa(G)} \prec \frac{1}{\log (1+\lambda^2)}.
\end{equation}
\end{theorem}

\begin{theorem}\label{thm:Condition}
Let $ H $ be an $ N \times N $ random matrix satisfying \eqref{e.Assumption1} and \eqref{e.Assumption2}. For any $ \eps>0 $, we have
\begin{equation}\label{eq:Condition_Number}
\P \pth{\frac{\kappa(G)}{N}>r + N^{-\frac{2}{3}+\eps}} - N^{-\frac{1}{3}-\eps} \leq \P \pth{ \frac{\kappa(H)}{N} > r } \leq \P \pth{\frac{\kappa(G)}{N}>r - N^{-\frac{2}{3}+\eps}} + N^{-\frac{1}{3}-\eps}
\end{equation}
\end{theorem}

\smallskip

As mentioned previously, the general hard edge case $ \lim N/M =1 $ without $ M \equiv N $ is a notoriously difficult problem in random matrix theory. To further generalize our results, we have the following slight extension towards the general case. 

\begin{theorem}\label{thm:General}
Let $ H $ be an $ M \times N $ random matrix satisfying \eqref{e.Assumption1} and \eqref{e.Assumption2} with $ M=N+O(N^{o(1)}) $. Then all results in Theorem \ref{thm:Smoothed_Singular}, Theorem \ref{t.Real}, Theorem \ref{thm:Smoothed_Condition} and Theorem \ref{thm:Condition} are still true.
\end{theorem}

\subsection{Outline of proofs}
The central idea of this paper is based on the Erd\H{o}s-Schlein-Yau dynamical approach in random matrix theory. In their seminal work \cite{erdos2012local}, the so-called \emph{three-step strategy} is developed to prove universality phenomena for random matrices. Roughly speaking, this framework is the following three ideas. 
\begin{enumerate}[(i)]
\item A priori estimates for spectral statistics. This is based on the analyzing the resolvent of the matrix, and such analysis is called local law in random matrix theory. Local law states that the spectral density converges to the limiting law on microscopic scale. This local law implies the eigenvalue rigidity phenomenon, which states that the eigenvalues are close to their typical locations. Such a priori control of the eigenvalue locations will play a significant role in further analysis.

\item Local relaxation of eigenvalues. This step is designed to prove universality for matrices with a tiny Gaussian component. We perturb the matrix by some independent Gaussian noise, and then under this perturbation, the dynamics of the eigenvalues is governed by the Dyson Brownian motion (DBM). Moreover, the spectral distribution of the Gaussian ensemble is the equilibrium measure of DBM. The ergodicity of DBM results in a fast convergence to the local equilibrium, and hence implies the universality for matrix with small Gaussian perturbation.

\item Density arguments. For any probability distribution of the matrix elements, there exists a distribution with a small Gaussian component (in the sense of Step (ii)) such that the two associated random matrices have asymptotically identical spectral statistics. Typically, such an asymptotic identity is guaranteed by some moment matching conditions and the comparison of resolvents.
\end{enumerate}
For a systematic discussion of this method, we refer to the monograph \cite{erdos2017book}. Following this strategy, our main techniques can also be divided into the following three steps.

\begin{itemize}
\item The first step is the local semicircle law for the symmetrization of the random matrix $ H $. This local law guarantees the optimal rigidity estimates for the singular values. This step is be based on classical works in random matrix theory such as \cite{bloemendal2014isotropic,erdos2014imprimitive,alt2017local}.

\item The second step is to interpolate the general matrix $ H $ with the Gaussian matrix $ G $, and estimate the dynamics of the singular value. More specifically, we consider the interpolation $ H_t = e^{-t/2}H + \sqrt{1-e^{-t}}G $, which solves the matrix Ornstein-Uhlenbeck stochastic differential equation
$$ \d H_t = \frac{1}{\sqrt{N}} \d B_t - \frac{1}{2}H_t \d t. $$
Note that this interpolation $ H_t $ is equivalent to the matrix perturbation in our smoothed analysis. We consider a weighted Stieltjes transform (defined in \eqref{eq:Weighted_Stieltjes}). A key innovation of our work is that, combined with a symmetrization trick, the evolution of the weighted Stieltjes along the dynamics of $ H_t $ satisfies a stochastic PDE that can be well approximated by a deterministic advection equation. This deterministic PDE yields a rough estimate for $ |\sigma_k(H_t)-\sigma_k(G)| $. Finally, using a delicate bootstrap argument, we show that the estimates for $ |\sigma_k(H_t)-\sigma_k(G)| $ are self-improving. Iterating the bootstrap argument to optimal scale, we derive the optimal smoothed analysis for the smallest singular value.

\item The last step is a quantitative resolvent comparison. In particular, the difference between the resolvent of two different random matrices are explicitly controlled in terms of the difference of their fourth moments. This comparison is proved via the Lindeberg exchange method. Together with the optimal smoothed analysis, this comparison theorem establishes the quantitative universality.
\end{itemize}


\subsection{Notations and paper organizations}
Throughout this paper, we denote $ C $ a generic constant which does not depend on any parameter but may vary form line to line. We write $ A \lesssim B $ if $ A \leq CB $ holds for some constant $ C $, and similarly write $ A \gtrsim B $ if $ A \geq C^{-1}B $. We also denote $ A \sim B $ if both $ C^{-1}B \leq A \leq C B $ hold. When $ A $ and $ B $ are complex valued, $ A \sim B $ means $ \re A \sim \re B $ and $ \im A \sim \im B $. We use $ \llbracket A,B \rrbracket := [A,B] \cap \Z $ to denote the set of integers between $ A $ and $ B $. We use $ a \vee b := \max\{a,b\} $ and $ a \wedge b := \min\{a,b\} $ to denote the maximum and minimum between $ a $ and $ b $, respectively.






\smallskip


The paper is organized as follows. In Section \ref{sec:Applications}, we discuss some applications of our results in numerical analysis and algorithms. In Section \ref{s.Dynamics}, we discuss the smoothed analysis for the smallest singular value of random matrices via the study of singular value dynamics. In Section \ref{s.Green}, we use the smoothed analysis to establish a full quantitative universality for the smallest singular value.  In Section \ref{sec:Condition}, we use the results on the smallest singular value to derive smoothed analysis and quantitative universality for the condition number. In Section \ref{sec:Non_Square}, we extend the result for square matrices to a slightly more general non-square case. Finally, in the Appendix, we collect some auxiliary results and provide the deferred technical proofs.


\section{Applications in Numerical Analysis and Algorithms}\label{sec:Applications}
In this section, we discuss some applications of our results in numerical analysis and algorithms. There are numerous circumstances where the smallest singular value and condition number play important roles. We do not intend to mention all of them and just focus on two simple scenarios in the framework of solving linear systems to illustrate the usefulness of our results. We expect our results can be applied in more complicated models and more advanced algorithms.

\subsection{Accuracy of least-square solution}
Consider the linear least-square optimization
$$ \min_{x \in \R^N} \|Ax-b\|_2^2, $$ 
where $ A $ is an $ M \times N $  matrix satisfying $ M-N=N^{o(1)} $, and $ b \in \R^N $ is a fixed vector. The loss of precision of this problem, denoted by $ \mathrm{LoP}(A,b) $, is the number of correct digits in the entries of the data $ (A,b) $ minus the same quantity for the computed solution. Let $ \mathrm{LoP}(A) $ denote the loss of precision for the worst $ b $. Then, as shown in \cite{higham2002accuracy}, we have
$$ \mathrm{LoP}(A) = \log M N^{3/2} + 2 \log \kappa(A) + O(1). $$
Let $ H $ be an $ M \times N $ random matrix satisfying \eqref{e.Assumption1} and \eqref{e.Assumption2}. Also let $ G $ be an $ M \times N $ Gaussian matrix. By Theorem \ref{thm:Smoothed_Condition} and Theorem \ref{thm:General}, for any $ \eps>0 $, with overwhelming probability we have
$$ \mathrm{LoP}(H+\lambda G) \leq \log M N^{3/2} +2 \log \kappa(G) + N^{-1+\eps} \frac{1}{ \log (1+\lambda^2)} + O(1). $$
Also, using Theorem \ref{thm:Condition} and \ref{thm:General}, for any $ \eps>0 $, with probability at least $ 1-N^{-1/3-\eps} $, we have
$$ \mathrm{LoP}(H) \leq \log M N^{3/2} + 2 \log \kappa(G) + N^{-\frac{2}{3}+\eps} + O(1) $$
The error terms are smaller than the $ O(1) $ term. These results imply that a general random matrix and its Gaussian perturbation can ensure accuracy as good as in the Gaussian case.

\subsection{Complexity of conjugate gradient method}
Consider the linear equation
$$ A^\top A x=c, $$
where $ A $ is an $ M \times N $ matrix with $ M-N = N^{o(1)} $, and $ c \in \R^N $ is a fixed vector. This linear system can be solved via the conjugate gradient algorithm. Let $ T_\delta(A) $ denote the needed iterations to obtain an $ \delta $-approximation of the true solution in the worst case. Then it is known (see e.g. \cite{trefethen1997numerical}) that
$$ T_\delta(A) = \frac{1}{2} \kappa(A) \delta. $$
Let $ H $ be an $ M \times N $ random matrix satisfying \eqref{e.Assumption1} and \eqref{e.Assumption2}. Also let $ G $ be an $ M \times N $ Gaussian matrix. By Theorem \ref{thm:Smoothed_Condition} and Theorem \ref{thm:General}, for any $ \eps>0 $, with overwhelming probability we have
$$ \abs{T_\delta(H+\lambda G) - T_\delta(G)} \leq  \frac{ \delta N^\eps}{\log (1+\lambda^2)} \lesssim N^\eps \frac{\delta}{\lambda^2} $$
This shows that as long as $ \lambda^2 \gg \delta $, the Gaussian perturbation $ H+\lambda G $ has time complexity as good as the Gaussian ensemble.

Similarly, using Theorem \ref{thm:Condition} and Theorem \ref{thm:General}, for any $ \eps>0 $, with probability at least $ 1-N^{-1/3-\eps} $, we have
$$ T_\delta(H) \leq T_\delta(G) + \delta N^{1/3 + \eps}. $$
This shows that as long as the required accuracy satisfies $ \delta \ll N^{-1/3} $, the time complexity for a general random matrix is as good as the Gaussian ensemble.

\section{Smoothed Analysis and Gaussian Approximation}\label{s.Dynamics}
\subsection{Singular value dynamics}

In smoothed analysis, we are interested in matrix perturbation of the form $ H+\lambda G $. After normalization of the variance, it is equivalent to study matrix of the form $ H_t=e^{-t/2}H + \sqrt{1-e^{-t}}G  $. More specifically, we have
\begin{equation}\label{eq:Correspondence}
H+\lambda G = \sqrt{1+\lambda^2} \pth{\frac{1}{\sqrt{1+\lambda^2}}H + \frac{\lambda}{\sqrt{1+\lambda^2}}G  } = \sqrt{1+\lambda^2} H_{\log(1+\lambda^2)}. 
\end{equation}
Let $ B $ be an $ N \times N $ matrix Brownian motion, i.e. $ B_{ij} $ are independent standard Brownian motions. Then the evolution of $ H_t $ is governed by the following matrix-valued Ornstein-Uhlenbeck process:
\begin{equation}\label{e.MatrixDBM}
\d H_t = \dfrac{1}{\sqrt{N}}\d B_t - \dfrac{1}{2} H_t \d t.
\end{equation}
Let $ \{s_k(t)\}_{k=1}^N $ denote the singular values of $ H_t $, then $ \{s_k(t)\}_{k=1}^N $ satisfy the following system of stochastic differential equations \cite[equation (5.8)]{erdos2012local},
\begin{equation}\label{e.SingularValueDBM}
\d s_k = \dfrac{\d B_k}{\sqrt{N}} + \left[ -\dfrac{1}{2}s_k +\dfrac{1}{2N} \sum_{\ell \neq k} \left( \dfrac{1}{s_k - s_\ell} + \dfrac{1}{s_k + s_\ell} \right) \right]\d t, \ \ \ 1 \leq k \leq N.
\end{equation}
To handle these SDEs, an important idea is the following symmetrization trick (see \cite[equation (3.9)]{che2019universality}):
$$ s_{-i}(t)=-s_i(t),\ \ \ B_{-i}(t)=-B_i(t),\ \ \forall t \geq 0,\ 1 \leq i \leq N. $$
With these notations, we label the indices from $ -1 $ to $ -N $ and $ 1 $ to $ N $, so that the zero index is omitted. Unless otherwise stated, this will be the convention and we will not emphasize it explicitly in the following parts of the paper. After symmetrization, for the real case we have
\begin{equation}\label{e.SymDBM}
\d s_k = \dfrac{\d B_k}{\sqrt{N}} + \left[ -\dfrac{1}{2}s_k +\dfrac{1}{2N} \sum_{\ell \neq \pm k} \dfrac{1}{s_k - s_\ell} \right]\d t, \ \ \ -N \leq k \leq N, k \neq 0.
\end{equation}

Now we use the coupling method introduced in \cite{landon2019fixed} to analyze these dynamics. Consider the interpotation between a general matrix $ H $ and a Gaussian matrix $ G $. Let $ \{\sigma_k(H)\}_{k=-N}^N $ and $ \{\sigma_k(G)\}_{k=-N}^N $ be the (symmetrized) singular values of $ H $ and $ G $, respectively. For $ \nu \in [0,1] $, define
$$ s_k^{(\nu)}(0) = (1-\nu) \sigma_k(H) + \nu \sigma_k(G). $$
With this initial condition, we denote the unique solution of \eqref{e.SymDBM} by $ \{s_k^{(\nu)}(t)\} $. Also, let $ \{\sigma_k(H,t)\} $ and $ \{\sigma_k(G,t)\} $ denote the solutions of \eqref{e.SymDBM} with initial conditions $ \{\sigma_k(H)\} $ and $ \{\sigma_k(G)\} $, respectively.

It is well known that the empirical measure of the eigenvalues of $ H^*H $ converges to the Marchenko-Pastur distribution
$$ \rho_{\MP}(x) = \dfrac{1}{2\pi}\sqrt{\dfrac{4-x}{x}}, $$
For $ 1 \leq k \leq N $, we define the typical position of the singular value $ \sigma_k $ as the quantile $ \gamma_k $ satisfying
$$  \int_{-\infty}^{\gamma_k^2} \rho_{\MP}(x)\d x = \dfrac{k}{N}. $$
We also define $ \gamma_{-k} = -\gamma_k $. By a change of variable, we have
\begin{equation}\label{e.typical}
\int_{-\infty}^{\gamma_k} \rho_\sc(x)\d x=\dfrac{N+k}{2N}, \ \ \ \int_{-\infty}^{\gamma_{-k}} \rho_\sc(x)\d x = \dfrac{N-k}{2N},
\end{equation}
where $ \rho_\sc(x) = \tfrac{1}{2\pi}\sqrt{(4-x^2)_+} $ is the semicircle law.

An important input of our proof is the following uniform rigidity estimates. For any fixed $ \eps>0 $, consider the set of good trajectories
\begin{equation}\label{e.Rigidity}
\A_\eps =\left\{ \left| s_k^{(\nu)}(t) - \gamma_k \right|<  N^{-\frac{2}{3}+\eps}(N+1-|k|)^{-\frac{1}{3}}\  \mbox{for all}\ 0 \leq t \leq 1, -N \leq k \leq N,0 \leq \nu \leq 1 \right\}.
\end{equation}
Such rigidity estimates for fixed $ t $ and $ \nu=0 $ or $ 1 $ were proved in \cite{erdos2014imprimitive,alt2017local,bloemendal2014isotropic,bourgade2014circular,cacciapuoti2013local}.The extension to uniform estimates in parameters can be done by a discretization argument: (1) discretize in $ t $ and $ \nu $; (2) use weyl's inequality to control increments over small time intervals; (3) use a maximum principle for the derivative with respect to the $ \nu $ parameter (see Lemma \ref{l.MaxPrin}) to control increments in small $ \nu $-intervals. 
As a consequence, we have

\begin{lemma}\label{l.Rigidity}
For any $ \eps>0 $, the event $ \A_\eps $ happens with overwhelming probability, i.e. for any $ D>0 $, there exists $ N_0(\eps,D) $ such that for any $ N>N_0 $ we have
$$ \P(\A_\eps)>1-N^{-D}. $$
\end{lemma}

We consider
$$ \varphi_k^{(\nu)}(t) := e^{\frac{t}{2}} \dfrac{\d}{\d \nu} s_k^{(\nu)}(t). $$
For the simplicity of notations, we omit the parameter $ \nu $ if the context is clear. Then $ \varphi_k $ satisfies the following non-local parabolic type equation.
\begin{equation}\label{e.Parabolic}
\dfrac{\d}{\d t}\varphi_k = \dfrac{1}{2N} \sum_{\ell \neq \pm k} \dfrac{\varphi_\ell - \varphi_k}{\left( s_\ell - s_k \right)^2}
\end{equation}
Let $ \psi_k=\psi_k^{(\nu)} $ solve the same equation as $ \varphi_k $ in \eqref{e.Parabolic} but with initial condition $ \psi_k(0)=|\varphi_k(0)|=|\sigma_k(H)-\sigma_k(G)| $. Following the same arguments in \cite[Lemma 3.1]{wang2019quantitative}, this equation satisfies a maximum principle.

\begin{lemma}\label{l.MaxPrin}
For all $ t \geq 0 $ and $ -N \leq k \leq N $, we have
$$ \psi_k(t)=\psi_{-k}(t),\ \ \ \psi_k(t) \geq 0, \ \ \ |\psi_k(t)| \leq \max_{-N \leq \ell \leq N} |\psi_\ell(0)|,\ \ \ |\varphi_k(t)| \leq \psi_k(t). $$
\end{lemma}

We consider the following weighted Stieltjes transform
\begin{equation}\label{eq:Weighted_Stieltjes}
\frakS_t(z) := e^{-\frac{t}{2}} \sum_{-N \leq k \leq N} \dfrac{\varphi_k(t)}{s_k(t) - z},\ \ \ \tilde{\frakS}_t(z) := e^{-\frac{t}{2}} \sum_{-N \leq k \leq N} \dfrac{\psi_k(t)}{s_k(t) - z}.
\end{equation}
Let $ S_t(z) $ and $ m_{\sc}(z) $ denote the Stieltjes transforms of the empirical measure for the singular values and of the semicircle law
$$ S_t(z)=\dfrac{1}{2N}\sum_{-N \leq k \leq N}\dfrac{1}{s_k -z},\ \ \ m_{\sc}(z)=\dfrac{-z+\sqrt{z^2-4}}{2}. $$
A well-known result in random matrix theory is the following local semicircle law for the Stieltjes transform $ S_t(z) $. Let $ \omega>0 $ be an arbitrarily fixed constant. Define the spectral domain
$$ \mathbf{S}=\mathbf{S}_\omega := \sth{z=E+\i \eta: |E| \leq \omega^{-1},N^{-1+\omega} \leq \eta \leq \omega^{-1}}. $$
For any $ \omega>0 $ and $ z \in \mathbf{S}_\omega $, it was shown in \cite{erdos2014imprimitive,alt2017local} that
\begin{equation}\label{eq:Local_SC_Law}
\abs{S_t(z) - m_\sc(z)} \prec \frac{1}{N\eta}.    
\end{equation}

As computed in \cite[Lemma 3.3]{wang2019quantitative}, a key result is that $ \frakS_t(z) $ and $  \tilde{\frakS}_t(z) $ satisfy the following stochastic advection equation.

\begin{lemma}\label{l.Dynamics}
For $ \im z \neq 0 $, we have
\begin{equation}\label{e.AdvectionEqn}
\begin{aligned}
\d \tilde{\frakS}_t &=\left(S_t(z)+\dfrac{z}{2}\right)(\partial_z \tilde{\frakS}_t)\d t+\dfrac{1}{4N} (\partial_{zz} \tilde{\frakS}_t) \d t+\left[\dfrac{e^{-\frac{t}{2}}}{2N}\sum_{-N \leq k \leq N} \dfrac{\psi_k}{(s_k-z)^2(s_k+z)}\right]\d t\\
& \quad -\dfrac{e^{-\frac{t}{2}}}{\sqrt{N}}\sum_{-N \leq k \leq N}\dfrac{\psi_k}{(s_k-z)^2}\d B_k.
\end{aligned}
\end{equation}
\end{lemma}

Based on the local semicircle law \eqref{eq:Local_SC_Law}, we expect that this stochastic differential equation can be approximated by the deterministic advection PDE
\begin{equation}\label{e.DeterministicPDE}
\partial_t h = \dfrac{\sqrt{z^2-4}}{2} \partial_z h.
\end{equation}
The above PDE have the following explicit characteristics
\begin{equation}\label{e.Characteristic}
z_t = \dfrac{e^{\frac{t}{2}}\left( z+\sqrt{z^2-4} \right) + e^{-\frac{t}{2}}\left( z-\sqrt{z^2-4} \right)}{2}.
\end{equation}
This implies $ \tilde{\frakS}_t(z) \approx \tilde{\frakS}_0(z_t) $, and we will justify this approximation in the next subsection. Moreover, we remark that $ \frakS_t $ satisfies the same equation \eqref{e.AdvectionEqn} with $ \psi_k $ replaced by $ \varphi_k $.

Before moving to the main estimates, we first collect some basic results, including the geometry of the characteristics and a rough estimate for the initial condition. These can be proved via direct computations and the details can be found in \cite[Section 2]{bourgade2021extreme}.

Let $ \xi(z)=\min\{|z-2|,|z+2|\} $. For any $ \eps>0 $, we consider the curve and the domain
$$ \S_\eps=\left\{ E+\i \eta : -2+ N^{-\frac{2}{3}+4\eps}<E<2- N^{-\frac{2}{3}+4\eps},\eta=N^{-1+4\eps}\xi(E)^{-\frac{1}{2}} \right\},\ \ \mathscr{R}_\eps = \bigcup_{0 \leq t \leq 1}\{z_t:z \in \S_\eps\}. $$
We also define $ a(z)=\dist(z,[-2,2]) $ and $ b(z)=\dist(z,[-2,2]^c) $.

\begin{lemma}\label{l.Geometry}
Uniformly in $ 0<t<1 $ and $ z=z_0 $ satisfying $ \eta=\im z>0 $ and $ |z-2|<\tfrac{1}{10} $, we have
$$ \re (z_t-z_0) \sim t\dfrac{a(z)}{\xi(z)^{1/2}}+t^2,\ \ \ \im (z_t-z_0) \sim t\dfrac{b(z)}{\xi(z)^{1/2}}. $$
In particular, for $ \eps>0 $, if $ z \in \S_\eps $, then $ z_t-z_0 \sim (t N^{-1+4\eps} \xi(E)^{-1} +t^2) + \i \xi(E)^{1/2}t $.

Moreover, for any $ \xi>0 $, and $ z=E+\i \eta \in [-2+\xi,2-\xi] \times [0,\xi^{-1}] $, we have $ \im (z_t-z_0) \sim t $.
\end{lemma}

\begin{lemma}\label{l.InitialCondition}
Let $ \eps>0 $ be any small constant. In the set $ \A_\eps $, for any $ z =E+\i\eta \in \mathscr{R}_\eps $, we have $ \im \tilde{\frakS}_0(z) \lesssim N^\eps $ if $ \eta > \max(E-2,-E-2) $, and $ \im\tilde{\frakS}_0(z) \lesssim N^\eps {\xi(z)^{-1}}{\eta} $ otherwise. The same bounds also hold for $ |\im \frakS_0| $.
\end{lemma}

A key ingredient of our proof is the following a priori estimate for $ \tilde{\frakS}_t $, whose proof is deferred to Appendix \ref{app:Apriori}.
\begin{proposition}\label{p.aPriori}


Let $ \eps>0 $ be any small constant. Uniformly for all $ 0<t<1 $ and $ z=E+\i \eta \in \S_\eps $, with overwhelming probability we have
$$ \im \tilde{\frakS}_t(z) \lesssim N^{2\eps} \dfrac{\xi(E)^{1/2}}{\left( \xi(E)^{1/2} \vee t \right)}. $$


\end{proposition}

This estimate yields a rough control for the decay of $ \varphi_k(t) $, which will be an important input for more refined estimates.

\begin{lemma}\label{lem:phi_decay_rough}


For all $ -N \leq k \leq N $ and $ 0\leq t \leq 1 $, we have
\begin{equation}\label{eq:phi_decay_rough}
|\varphi_k(t)| \prec \frac{1}{N} \dfrac{1}{\left((\frac{N+1-|k|}{N})^{1/3} \vee t\right)}
\end{equation}
\end{lemma}
\begin{proof}
By Lemma \ref{l.MaxPrin}, it suffices to control $ \psi_k(t) $. 
By the nonnegativity of $ \psi_k(t) $, we have
$$ \im \tilde{\frakS}_t(z) = \sum_{-N \leq k \leq N} \frac{\psi_k(t) \im z}{|s_k(t)-z|^2} \geq \psi_k(t) \frac{\im z}{|s_k(t)-z|^2}, $$
which implies
$$ \psi_k(t) \leq \im \tilde{\frakS}_t(z)\dfrac{|s_k(t)-z|^2}{\im z}. $$
Let $ \eps>0 $. For $ (N+1-|k|)>N^{10 \eps} $, pick the point $ z=\gamma_k + \i {N^{-1+4\eps} \xi(\gamma_k)^{-1/2}} \in \S_\eps $. In this case we have $ {\xi(\gamma_k)}^{1/2} \sim (\tfrac{N+1-|k|}{N})^{1/3} $. Therefore, in the set $ \A_\eps $, by Proposition \ref{p.aPriori}, uniformly for all $ -N+N^{10\eps} \leq k \leq N-N^{10\eps} $ and $ 0 \leq t \leq 1 $, with overwhelming probability we have
\begin{equation*}
|\psi_k(t)| < \dfrac{N^{8\eps}}{N}\dfrac{1}{\left((\frac{N+1-|k|}{N})^{1/3} \vee t\right)}.
\end{equation*}
For $ (N+1-|k|) \leq N^{10 \eps} $, without loss of generality we consider $ N+1-k \leq N^{10 \eps} $. In this case, let $ k_0=N-N^{10\eps}+1 $ and consider $ z=\gamma_{k_0} + \i {N^{-1+4\eps} \xi(\gamma_{k_0})^{-1/2}} $. The same argument results in a similar bound with a larger $ N^{20 \eps} $ factor. By the arbitrariness of $ \eps $, this completes the proof.
\end{proof}


\subsection{Local relaxation at the hard edge}
In this subsection we prove a quantitative estimate for the local relaxation flow \eqref{e.SymDBM} at the hard edge. The main estimate in this section is the following. We remark that Theorem \ref{t.Relaxation} is equivalent to Theorem \ref{thm:Smoothed_Singular} using the rescaling \eqref{eq:Correspondence}.

%

\begin{theorem}\label{t.Relaxation}
For $ \eps_0>0 $ arbitrarily small and any $ N^{-1+\eps_0}<t<1 $, we have
\begin{equation}\label{e.Relaxation}
|\sigma_1(H,t)-\sigma_1(G,t)| \prec \dfrac{1}{N^2 t}.
\end{equation}
\end{theorem}

To estimate $ \varphi_k(t) $ near the hard edge, we introduce the following quantity to approximate it. Let $ \gamma_k^t=(\gamma_k)_t $ with the convention $ \gamma^t=(\gamma+\i 0^+)_t $, and define
\begin{equation}\label{eq:Def_hatphi}
\uu_k(t) := \dfrac{1}{2N \im m_{\sc}(\gamma_k^t)} \sum_{-N \leq j \leq N} \im \left( \dfrac{1}{\gamma_j - \gamma_k^t} \right) (\sigma_j(H)-\sigma_j(G)).
\end{equation}

Our goal is to prove the following estimates

\begin{proposition}\label{p.Homogenization}
Let $ 0<c<1 $ be a fixed small constant. For $ \eps_0>0 $ arbitrarily small with any $ N^{-1+\eps_0}<t<1 $ and $ k \in \llbracket (c-1)N,(1-c)N \rrbracket $, we have
$$ |\sigma_k(H,t)-\sigma_k(G,t)-\uu_k(t)| \prec \dfrac{1}{N^2 t}. $$
\end{proposition}

To obtain the optimal control for the local relaxation flow, we need to carefully estimate $ \uu_k $ near the hard edge. A first step towards such estimates is given in the following lemma.

\begin{lemma}\label{l.ApproxUbar}
Let $ \eps>0 $ and $ 0<c<1 $. For any $ (k,\ell) \in \llbracket (c-1)N,(1-c)N \rrbracket ^2 $, $ |E|<2-c $, and $ s,t,\eta \in [N^{-1+4\eps},1] $, in the set $ \A_\eps $, for $ z=E+\i\eta $ we have
\begin{equation}\label{e.GapUbar}
\left| \uu_k(t) - \uu_\ell(s) \right| \lesssim N^{\eps} \left( \dfrac{|k-\ell|}{N^2 (s \wedge t)}+\dfrac{|s-t|}{N (s \wedge t)} \right),
\end{equation}
\begin{equation}\label{e.ApproxUbar}
\left| \dfrac{\im \frakS_0(z_t)}{2N \im S_0(z_t)} - \uu_\ell(s) \right| \lesssim N^{\eps} \left( \dfrac{|E-\gamma_\ell|}{N (s \wedge (\eta+t))} + \dfrac{|\eta+t-s|}{N (s \wedge (\eta+t))} \right).
\end{equation}
\end{lemma}

\begin{proof}
By the properties of the Stieltjes transform $ m_\sc(z) $ (see e.g. \cite[Section 6]{erdos2017book}) and direct computation, we have
$$ \im m_\sc(\gamma_k^t) \gtrsim 1,\ \  \im m_\sc(\gamma_\ell^s) \gtrsim 1, $$
and
\begin{equation}\label{eq:Diff_Im_msc}
\abs{\im m_\sc(\gamma_k^t) - \im m_\sc(\gamma_\ell^s)} + \abs{\gamma_k^t - \gamma_\ell^s} \lesssim \frac{|k-\ell|}{N} + |s-t|.
\end{equation}
In the set $ \A_\eps $, the rigidity estimates imply that
$$ |\sigma_j(H) - \sigma_j(G)| \leq N^{-\frac{2}{3}+\eps} (N+1-|j|)^{-\frac{1}{3}}. $$
Then we have
\begin{equation}\label{eq:Im_Int_Approx1}
\begin{aligned}
\abs{\frac{1}{2N} \im \pth{\sum_{-N \leq j \leq N} \frac{\sigma_j(H) - \sigma_j(G)}{\gamma_j - \gamma_k^t}}} &\lesssim \frac{1}{2N} (\im \gamma_k^t) \sum_{-j \leq j \leq N} \frac{N^{-\frac{2}{3}+\eps} (N+1-|j|)^{-\frac{1}{3}}}{(\gamma_j - \re \gamma_k^t)^2 + (\im \gamma_k^t)^2}\\
&\lesssim N^{-1+\eps} (\im \gamma_k^t) \int_{-2}^2 \frac{\xi(x)^{-\frac{1}{2}}}{(x-\re \gamma_k^t)^2+(\im \gamma_k^t)^2} \d \rho_\sc(x)\\
&\lesssim N^{-1+\eps} (\im \gamma_k^t) \int_{-2}^2 \frac{1}{(x-\re \gamma_k^t)^2+(\im \gamma_k^t)^2} \d x\\
&\lesssim N^{-1+\eps}.
\end{aligned}
\end{equation}
By triangle inequality, we have
\begin{align*}
\abs{\uu_k(t) - \uu_\ell(s)} & \leq \frac{1}{2N} \abs{ \pth{ \frac{1}{\im m_\sc(\gamma_k^t)} - \frac{1}{\im m_\sc(\gamma_\ell^s)} } \im \pth{\sum_{-N \leq j \leq N} \frac{\sigma_j(H) - \sigma_j(G)}{\gamma_j - \gamma_k^t}} }\\
&\quad + \frac{1}{2N \im m_\sc(\gamma_\ell^s)} \sum_{-N \leq j \leq N} \abs{ \im \pth{\frac{1}{\gamma_j - \gamma_k^t}} - \im \pth{\frac{1}{\gamma_j - \gamma_\ell^s}} } \abs{\sigma_j(H) - \sigma_j(G)}\\
&=: I_1+I_2.
\end{align*}
Using \eqref{eq:Diff_Im_msc} and \eqref{eq:Im_Int_Approx1}, we obtain
$$ I_1 \lesssim N^\eps \pth{\frac{|k-\ell|}{N^2} + \frac{|s-t|}{N}}. $$
For the second term $ I_2 $, in the set $ \A_\eps $, the rigidity and \eqref{eq:Diff_Im_msc} imply that
\begin{align*}
I_2 &\lesssim N^{-1+\eps} \sum_{-N \leq j \leq N} N^{-\frac{2}{3}} (N+1-|j|)^{-\frac{1}{3}} \abs{ \frac{\gamma_k^t - \gamma_\ell^s}{(\gamma_j - \gamma_k^t)(\gamma_j - \gamma_\ell^s)} }\\
&\lesssim N^{-1+\eps} \pth{\frac{|k-\ell|}{N} + |s-t|} \sum_{-N \leq j \leq N} N^{-\frac{2}{3}} (N+1-|j|)^{-\frac{1}{3}} \pth{\frac{1}{|\gamma_j-\gamma_k^t|^2} + \frac{1}{|\gamma_j - \gamma_\ell^s|^2}}.
\end{align*}
Recall from Lemma \ref{l.Geometry} that $ \im \gamma_k^t \sim t $ and $ \im \gamma_\ell^s \sim s $. Using a similar argument as in \eqref{eq:Im_Int_Approx1} we obtain
$$ I_2 \lesssim N^{-1+\eps} \pth{\frac{|k-\ell|}{N} + |s-t|} \pth{\frac{1}{t} + \frac{1}{s}} \lesssim N^{\eps} \pth{\frac{|k-\ell|}{N^2 (s \wedge t)} + \frac{|s-t|}{N(s \wedge t)}}.  $$
Hence we have proved \eqref{e.GapUbar}. For the other part \eqref{e.ApproxUbar}, it can be proved via the same arguments.
\end{proof}

As a consequence, we have a good control for the size of $ \uu_k(t) $ away from the soft edge. This is based on the symmetric structure of $ \{\uu_k\} $.

\begin{lemma}\label{l.SizeUbar}
Let $ \eps>0 $ and $ 0<c<1 $. For any $ t \in [N^{-1+4\eps},1] $ and $ k \in \llbracket (c-1)N,(1-c)N \rrbracket $, with overwhelming probability we have
\begin{equation}\label{e.SizeUbar}
|\uu_k(t)| \lesssim N^\eps \dfrac{k}{N^2 t}.
\end{equation}
\end{lemma}
\begin{proof}
A key observation is the following
$$ \re \gamma_{-k}^t=-\re\gamma_k^t,\ \ \im \gamma_{-k}^t=\im \gamma_k^t,\ \ \re m_{\sc}(\gamma_{-k}^t) =-\re m_{\sc}(\gamma_{k}^t) ,\ \ \im m_{\sc}(\gamma_{-k}^t) =\im m_{\sc}(\gamma_{k}^t).  $$
Therefore, we have
\begin{align*}
\uu_{-k}(t) &= \dfrac{1}{2N \im m_{\sc}(\gamma_{-k}^t)} \sum_{-N \leq j \leq N} \im \left( \dfrac{1}{\gamma_j - \gamma_{-k}^t} \right) (\sigma_j(H) - \sigma_j(G))\\
&= \dfrac{1}{2N \im m_{\sc}(\gamma_k^t)} \sum_{-N \leq j \leq N} \dfrac{\im(\gamma_{-k}^t)}{\left(\gamma_j - \re(\gamma_{-k}^t) \right)^2+(\im \gamma_{-k}^t)^2} (\sigma_j(H) - \sigma_j(G))\\
&= \dfrac{1}{2N \im m_{\sc}(\gamma_k^t)} \sum_{-N \leq j \leq N} \dfrac{\im(\gamma_{k}^t)}{\left(\gamma_j + \re(\gamma_{k}^t) \right)^2+(\im \gamma_{k}^t)^2} (\sigma_j(H) - \sigma_j(G))
\end{align*}
Using the symmtrization of the singular values, we further have
\begin{align*}
\uu_{-k}(t) &= -\dfrac{1}{2N \im m_{\sc}(\gamma_k^t)} \sum_{-N \leq j \leq N} \dfrac{\im(\gamma_{k}^t)}{\left(\gamma_{-j} - \re(\gamma_{k}^t) \right)^2+(\im \gamma_{k}^t)^2} (\sigma_{-j}(H)-\sigma_{-j}(G))\\
&= -\dfrac{1}{2N \im m_{\sc}(\gamma_k^t)} \sum_{-N \leq j \leq N} \im \left( \dfrac{1}{\gamma_{-j} - \gamma_{k}^t} \right) (\sigma_{-j}(H)-\sigma_{-j}(G))\\
&= -\dfrac{1}{2N \im m_{\sc}(\gamma_k^t)} \sum_{-N \leq j \leq N} \im \left( \dfrac{1}{\gamma_j - \gamma_k^t} \right) (\sigma_j(H) - \sigma_j(G))\\
&= - \uu_k(t)
\end{align*}
Consequently, by \eqref{e.GapUbar} we have
$$ |\uu_k(t) - \uu_{-k}(t)| = 2|\uu_k(t)| \lesssim N^\eps \dfrac{2k}{N^2 t}. $$
This shows the desired result.
\end{proof}

Finally, it's straightfoward to derive Theorem \ref{t.Relaxation} from Proposition \ref{p.Homogenization} and Lemma \ref{l.SizeUbar}. Thus, our primary goal is to prove Proposition \ref{p.Homogenization}.

\subsection{Bootstrap arguments}
We will prove the main technical estimate Proposition \ref{p.Homogenization} via a bootstrap argument. 

\begin{definition}[Hypothesis $ \calH_{\alpha} $]
Consider the following hypothesis: For any fixed small $ 0<c<1 $, the following holds for $ \eps_0>0 $ arbitrarily small. For any $ N^{-1+\eps_0}<t<1 $, $ k \in \llbracket (c-1)N,(1-c)N \rrbracket $ and $ \nu \in [0,1] $, we have
\begin{equation}\label{e.Pa}
\left| \varphi_k^{(\nu)}(t) - \uu_k(t) \right| \prec \dfrac{(Nt)^\alpha}{N^2 t}.
\end{equation}
\end{definition}

Proposition \ref{p.Homogenization} is derived via a bootstrap of the hypothesis $ \calH_{\alpha} $. Specifically, we have the following two lemmas.

\begin{lemma}\label{l.InitialIteration}
The hypothesis $ \calH_1 $ is true.
\end{lemma}
\begin{proof}
Recall from Lemma \ref{lem:phi_decay_rough} that \eqref{eq:phi_decay_rough} implies that $ \varphi_k^{(\nu)}(t) \prec N^{-1} $. On the other hand, from the definition \eqref{eq:Def_hatphi} of $ \uu_k $, using the rigidity and \eqref{eq:Im_Int_Approx1} we obtain $ \uu_k(t) \prec N^{-1} $ thanks to the arbitrariness of $ \eps_0 $. Therefore, the triangle inequality yields $ |\varphi_k^{(\nu)}(t) - \uu_k(t)| \prec N^{-1} $, which completes the proof.
\end{proof}

\begin{lemma}\label{l.Bootstrap}
If $ \calH_{\alpha} $ is true, then $ \calH_{3\alpha/4} $ is true, i.e.
$$ \left| \varphi_k^{(\nu)}(t) - \uu_k(t) \right| \prec \dfrac{(Nt)^{\frac{3\alpha}{4}}}{N^2 t}. $$
\end{lemma}

The self-improving property of the hypothesis $ \calH_\alpha $ stated in Lemma \ref{l.Bootstrap} is the main technical part of the proof for Proposition \ref{p.Homogenization}. We defer its proof to Appendix \ref{app:Bootstrap}.

Finally, the optimal control \eqref{e.Relaxation} for the local relaxation flow at the hard edge follows from these two lemma together with Lemma \ref{l.SizeUbar}.
\begin{proof}[Proof of Proposition \ref{p.Homogenization}]
Note that
\begin{equation}\label{e.S_kR_kU_k}
\sigma_k(H,t)-\sigma_k(G,t) - \uu_k(t)= \int_0^1 \varphi_k^{(\nu)}(t) \d \nu - \uu_k(t) = \int_0^1 \left( \varphi_k^{(\nu)}(t) - \uu_k(t) \right) \d \nu
\end{equation}
Consider an arbitrarily fixed $ \delta>0 $, based on Lemma \ref{l.InitialIteration} and Lemma \ref{l.Bootstrap}, after a finite time of iterations, with overwhelming probability we have
$$ \left| \varphi_k(t) - \uu_k(t) \right| \leq \dfrac{N^\delta}{N^2 t} $$
This shows that for any fixed $ \tilde{D} $ and $ p $, and for large enough $ N $, we have
$$ \E\left( \left| \varphi_k(t) - \uu_k(t) \right|^{2p} \right) \leq \left(  \dfrac{N^{\delta}}{N^2 t} \right)^{2p} + N^{-\tilde{D}}. $$
By \eqref{e.S_kR_kU_k} we obtain
\begin{equation*}
\E \left( \left| \sigma_k(H,t)-\sigma_k(G,t) - \uu_k(t) \right|^{2p} \right) \leq \int_0^1 \E\left( \left| \varphi_k(t) - \uu_k(t) \right|^{2p} \right) \d \nu\\
\leq \left( \dfrac{N^{\delta}}{N^2 t} \right)^{2p} + N^{-\tilde{D}}.
\end{equation*}
We choose $ p=D/\delta $ and $ \tilde{D}=D+100p $, and then the Markov inequality yields
$$ \P\left( \left| \sigma_k(H,t)-\sigma_k(G,t) - \uu_k(t) \right| \leq \dfrac{N^{2\delta}}{N^2 t} \right) >1-N^{-D}, $$
which completes the proof thanks to the arbitrariness of $ \delta $ and $ D $.
\end{proof}

\section{Quantitative Universality}\label{s.Green}

\subsection{Quantitative resolvent comparison}
In classical random matrix theory, the spectral universality is proved by comparison of the resolvent for matrices with some moment matching conditions. To obtain a quantitative universality for the smallest singular value, we need a quantitative version of the resolvent comparison theorem.

For a fixed constant $ a \in (1,2) $, let $ \rho=\rho(N) \in [N^{-a},N^{-1}] $ be a cutoff scaling. Let $ r>0 $ and consider two symmetric functions $ f_1(x),f_2(x)  $ that are non-increasing in $ |x| $, given by
\begin{equation*}
f_1(x) :=
\begin{cases}
0 & \textrm{if } |x|>rN^{-1}\\
1 & \textrm{if } |x|<rN^{-1}-\rho
\end{cases}
,\ \ \ f_2(x) :=
\begin{cases}
0 & \textrm{if } |x|>rN^{-1}+\rho\\
1 & \textrm{if } |x|<rN^{-1}
\end{cases}
.
\end{equation*}
Also, consider a fixed non-increasing smooth function $ F $ such that $ F(x)=1 $ for $ x \leq 0 $ and $ F(x)=0 $ for $ x \geq 1 $.

A key observation is that the functions $ f_1,f_2 $ and $ F $ can bound the distribution of the smallest singular value $ \sigma_1(H) $. For any function $ f: \R \to \R $, we denote $ \Tr f(H) := \sum_{i=-N}^N f(\sigma_i(H)) $.

\begin{lemma}\label{lem:Distribution_Bound}
We have
\begin{equation}\label{eq:Distribution_Bound}
\E \qth{F \pth{\Tr f_2(H)}} \leq \P \pth{\sigma_1(H)>rN^{-1}} \leq \E \qth{F \pth{\Tr f_1(H)}},
\end{equation}
\end{lemma}

\begin{proof}
For the right-hand side, assume $ \sigma_1(H) >rN^{-1} $. By definition of the function $ f_1 $, we have $ \sum_{i=-N}^N f_1(\sigma_i(H))=0 $, which implies $ F(\Tr f_1(H))=1 $. Also note that $ F \geq 0 $. Therefore, we conclude $ \1\sth{\sigma_1(H)>rN^{-1}} \leq F(\Tr f_1(H)) $ and this yields
$$ \P \pth{\sigma_1(H) > rN^{-1}} = \E \qth{\1\sth{\sigma_1(H)>rN^{-1}}} \leq \E \qth{F(\Tr f_1(H))}. $$
The left-hand side can be proved similarly.
\end{proof}

When estimating the distribution $ \P \pth{\sigma_1(H) > rN^{-1}} $, thanks to the rigidity of singular values, we can assume $ r<N^\eps $ without loss of generality, where $ \eps>0 $ is a constant that can be arbitrarily small. Based on Lemma \ref{lem:Distribution_Bound}, to compare the distribution of the smallest singular values of different random matrices, it suffices to compare the functions $ \Tr f_1 $ and $ \Tr f_2 $. In the remaining part of this section, we provide a systematic treatment of such a comparison.

Pick a point $ E \in \R $ with $ 0<E< N^{-1+\eps} $. Let $ f(x) $ be a smooth symmetric function that is non-increasing in $ |x| $ satisfying
\begin{equation}
f(x)=
\begin{cases}
0 & \textrm{if } |x|>E\\
1 & \textrm{if } |x|<E-\rho
\end{cases}
,\ \ \ \textrm{and }\  \|f^{(k)}\|_\infty \lesssim \rho^{-k}\ \textrm{for } k=1,2.
\end{equation}
For the test functions $ f $ and $ F $ defined as above, we have the following quantitative comparison of the resolvents, whose proof is deferred to Appendix \ref{app:Comparison}.

\begin{proposition}\label{prop:Resolvent_Comparison}
Let $ X $ and $ Y $ be two independent random matrices satisfying \eqref{e.Assumption1} and \eqref{e.Assumption2}. Assume the first three moments of the entries $  $ are identical, i.e. $ \E[X_{ij}^k]=\E[Y_{ij}^k] $ for all $ 1 \leq i,j \leq N $ and $ 1 \leq k \leq 3 $. Suppose also that for some parameter $ t=t(N) $ we have
\begin{equation}\label{eq:Fourth_Moment_Diff}
\left| \E [(\sqrt{N}X_{ij})^4] - \E [(\sqrt{N}Y_{ij})^4] \right| \leq t,\ \ \textrm{for all } 1 \leq i,j \leq N.
\end{equation}
With the test functions $ f $ and $ F $ defined as above, there exists a constant $ C>0 $ such that the following is true for any $ \eps>0 $
\begin{equation}\label{eq:Resolvent_Comparison}
\left| \E [F(\Tr f(X))] - \E [F(\Tr f(Y))] \right| \leq  N^{C\eps} \pth{\frac{1}{\rho N^2} + \frac{\pth{\rho N^a}^5}{\sqrt{N}} + t \rho N^a}.  
\end{equation}
\end{proposition}

\subsection{Proof of Theorem \ref{t.Real}}
Using the quantitative comparison theorem (Proposition \ref{prop:Resolvent_Comparison}) and the smoothed analysis (Theorem \ref{t.Relaxation}), we now prove the quantitative universality.

For a general random matrix $ H $ satisfying Assumptions \eqref{e.Assumption1} and \eqref{e.Assumption2}, there exists another matrix $ H_0' $ that also satisfies the same assumptions such that the matrix $ H_t' := e^{-t/2}H_0' + (1-e^{-t})^{1/2} G $ has the same first three moments as $ H $ and the difference between the fourth moments (in the sense of \eqref{eq:Fourth_Moment_Diff}) is $ O(t) $. This is guaranteed by \cite[Lemma 3.4]{erdos2011universality}.

Lemma \ref{lem:Distribution_Bound} and Proposition \ref{prop:Resolvent_Comparison} yields
\begin{multline}
\E \qth{ F(\Tr f_2(H_t')) } - N^{C\eps} \pth{\frac{1}{\rho N^2} + \frac{\pth{\rho N^a}^5}{\sqrt{N}} + t \rho N^a} \leq \P \pth{\sigma_1(H)>rN^{-1}}\\
\leq \E \qth{ F(\Tr f_1(H_t')) } + N^{C\eps} \pth{\frac{1}{\rho N^2} + \frac{\pth{\rho N^a}^5}{\sqrt{N}} + t \rho N^a}.
\end{multline}
Using Lemma \ref{lem:Distribution_Bound} for $ H_t' $ with $ f_1 $ and $ f_2 $ shifted by $ \pm \rho $, we have
\begin{multline}
\P \pth{\sigma_1(H_t') > \frac{r}{N} + \rho} - N^{C\eps} \pth{\frac{1}{\rho N^2} + \frac{\pth{\rho N^a}^5}{\sqrt{N}} + t \rho N^a} \leq \P \pth{\sigma_1(H)>\frac{r}{N}}\\
\leq \P \pth{\sigma_1(H_t') > \frac{r}{N} - \rho} + N^{C\eps} \pth{\frac{1}{\rho N^2} + \frac{\pth{\rho N^a}^5}{\sqrt{N}} + t \rho N^a}.
\end{multline}
Using smoothed analysis Theorem \ref{t.Relaxation}, we have
\begin{multline}
\P \pth{\sigma_1(G) > \frac{r}{N} + \rho + \frac{1}{N^2 t}} - N^{C\eps} \pth{\frac{1}{\rho N^2} + \frac{\pth{\rho N^a}^5}{\sqrt{N}} + t \rho N^a} \leq \P \pth{\sigma_1(H)>\frac{r}{N}}\\
\leq \P \pth{\sigma_1(G) > \frac{r}{N} - \rho - \frac{1}{N^2 t}} + N^{C\eps} \pth{\frac{1}{\rho N^2} + \frac{\pth{\rho N^a}^5}{\sqrt{N}} + t \rho N^a}.
\end{multline}
Taking $ \rho=N^{-a} $, $ t=N^{a-2} $ and setting $ a=1+\delta $, we obtain the optimal bounds
\begin{multline}\label{eq:Universality_Smallest_Parameter}
\P \pth{ N\sigma_1(G) > r + N^{-\delta} } - N^{C\eps} \pth{N^{-1+\delta} \vee N^{-\frac{1}{2}} } \leq \P \pth{N \sigma_1(H)>r}\\
\leq \P \pth{ N\sigma_1(G) > r - N^{-\delta} } + N^{C\eps} \pth{N^{-1+\delta} \vee N^{-\frac{1}{2}} }.
\end{multline}
Hence, thanks to the arbitrariness of $ \eps $, we have proved Theorem \ref{t.Real}.

\smallskip

Finally, for the complex case, using the exact formula for the distribution of $ \sigma_1(G_\C) $, we obtain a rate of convergence to the limiting law. Recall that
$$ \P(N \sigma_1(G_\C) \leq r) =\int_0^r e^{-x} \d x=1-e^{-r}. $$
\begin{proof}[Proof of Corollary \ref{c.Complex}]
For the complex case, the previous arguments are still valid. Therefore, we still have \eqref{eq:Universality_Smallest_Parameter}. Since $ N \sigma_1(G_\C) $ has a bounded density, we have
\begin{multline} 
\P \pth{N \sigma_1(G_\C) \leq r} - N^\eps \pth{N^{-\delta} + (N^{-1+\delta} \vee N^{-1/2})} \leq \P \pth{ N \sigma_1(H_\C) \leq r }\\
\P \pth{N \sigma_1(G_\C) \leq r} + N^\eps \pth{N^{-\delta} + (N^{-1+\delta} \vee N^{-1/2})}
\end{multline}
Choosing $ \delta=1/2 $, we obtain the optimal estimate
$$ \P \pth{N \sigma_1(H_\C)  \leq r} = 1-e^{-r} + N^{-\frac{1}{2}+\eps}, $$
which proves the desired result.
\end{proof}

\section{Condition Number}\label{sec:Condition}
\subsection{Smoothed analysis}
Note that the condition number $ \kappa(H) $ is scaling invariant in the sense that $ \kappa(a H) = \kappa(H) $ for any $ a>0 $. Therefore, in the smoothed analysis, it suffices to consider $ H_t = e^{-t/2}H + (1-e^{t})^{1/2} G $, whose singular values satisfies the stochastic differential equation \eqref{e.SingularValueDBM}.

Recall from Theorem \ref{t.Relaxation}, we have shown that
$$ |\sigma_1(H_t) - \sigma_1(G)| \prec \frac{1}{N^2 t}. $$
Note that in Lemma \ref{lem:phi_decay_rough}, using the same arguments as in Proposition \ref{p.Homogenization}, we can derive that
$$ |\sigma_N(H_t) - \sigma_N(G)| \prec \frac{1}{Nt}. $$
Then for any large $ D>0 $ and small $ \eps_1,\eps_2>0 $, there exists $ N_0(\eps_1,\eps_2,D) $ such that the following holds with probability at least $ 1-N^{-D} $.
$$ \sigma_1(G) - \frac{N^{\eps_1}}{N^2t} \leq \sigma_1(H_t) \leq \sigma_1(G) + \frac{N^{\eps_1}}{N^2t},\ \ \ \sigma_N(G) - \frac{N^{\eps_2}}{Nt} \leq \sigma_N(H_t) \leq \sigma_N(G) + \frac{N^{\eps_2}}{Nt}. $$
Without loss of generality, we assume that $ \eps_1 \sim \eps_2 \sim \eps $ for some $ \eps>0 $ that can be arbitrarily small. Then we have
$$ \frac{\kappa(H_t)}{N} = \frac{\sigma_N(H_t)}{N \sigma_1(H_t)} \leq \frac{\sigma_N(G) + \frac{N^{\eps_2}}{Nt}}{N \sigma_1(G) - \frac{N^{\eps_1}}{Nt}} \leq \frac{\kappa(G)}{N} + \frac{N^{C\eps}}{Nt}, $$
where in the last inequality we use that $ N \sigma_1(G) $ is of order constant with overwhelming probability. By the arbitrariness of $ \eps>0 $, we can relabel the parameter and then obtain
$$ \kappa(H_t) \leq \kappa(G) + \frac{N^{\eps}}{t}. $$
Similarly, we can also prove a lower bound and conclude that
$$ |\kappa(H_t) - \kappa(G)| \prec \frac{1}{t}. $$
For the matrix $ H+\lambda G $, we can write the matrix as
$$ H+\lambda G = \sqrt{1+\lambda^2} \pth{\frac{1}{\sqrt{1+\lambda^2}} H + \frac{\lambda}{\sqrt{1+\lambda^2}}G  } = \sqrt{1+\lambda^2} H_{\log (1+\lambda^2)}. $$
Therefore we have that $ \kappa(H+\lambda G) = \kappa(H_{\log(1+\lambda^2)}) $. As a consequence, we obtain that
$$ |\kappa(H+\lambda G) - \kappa(G)| \prec \frac{1}{\log (1+\lambda^2)}, $$
which completes the proof for Theorem \ref{thm:Smoothed_Condition}.

\subsection{Quantitative universality}
In this section, we prove Theorem \ref{thm:Condition}, which establishes the quantitative universality for the condition number. We will use the universality for both the smallest singular value and the largest singular values. In particular, since the smallest singular values is typically of order $ O(N^{-1}) $, a quantitative control is necessary. The largest singular value is of order constant, and therefore it is easier to deal with. Specifically, for the largest singular value, due to the Tracy-Widom limit for the largest eigenvalue of the sample covariance matrix, we know that $ |\sigma_N(H)-2| \leq  N^{-2/3+\eps}  $ for any $ \eps>0 $ with overwhelming probability.

Let $ \eps>0 $ be an arbitrarily small constant. For any large $ D>0 $ and sufficiently large $ N $, we have
$$ \P \pth{ \frac{\kappa(H)}{N} >r } = \P \pth{N \sigma_1(H) < r^{-1} \sigma_N(H)} \leq \P \pth{ N \sigma_1(H) < r^{-1}(2 +  N^{-\frac{2}{3}+\eps}) } +N^{-D} $$
Using Theorem \ref{t.Real}, we obtain
\begin{align*}
\P \pth{ N^{-1} \kappa(H) >r } &\leq  \P \pth{rN \sigma_1(G) < 2} + N^{-\frac{1}{3}-\eps} +N^{-D}\\
&\leq \P \pth{ rN \sigma_1(G) < \sigma_N(G)+  N^{-\frac{2}{3}+\frac{\eps}{2}} } + N^{-\frac{1}{3}-\eps}+N^{-D}\\
&\leq \P \pth{ \frac{\kappa(G)}{N}>r - \frac{N^{\eps/2}}{N^{2/3}} \frac{1}{N \sigma_1(G)} } + N^{-\frac{1}{3}-\eps} + N^{-D}\\
&\leq \P \pth{ \frac{\kappa(G)}{N} > r - N^{-\frac{2}{3}+\eps} } + N^{-\frac{1}{3}-\eps} + N^{-D}
\end{align*}
where the third inequality follows from that $ N \sigma_1(G) $ is of order constant with overwhelming probability. Taking large enough $ D $, we have
$$ \P \pth{ \frac{\kappa(H)}{N} > r } \leq \P \pth{ \frac{\kappa(G)}{N} >r - N^{-\frac{2}{3}+\eps}} + O \pth{ N^{-\frac{1}{3}-\eps} }. $$
This yields the upper bound in \eqref{eq:Condition_Number}, and the lower bound can be proved similarly. Hence, we have proved Theorem \ref{thm:Condition}.

\section{Beyond Strictly Square Matrices}\label{sec:Non_Square}

As mentioned in the Introduction, the optimal local law for an $ M \times N $ random matrix with general $ \lim N/M =1 $ is a notoriously hard problem. In particular, the optimal rigidity estimates Lemma \ref{l.Rigidity} is unknown unless we restrict $ M \equiv N $. In this section, we discuss a slight extension of the strictly-square case. We show that in the regime $ M=N+O(N^{o(1)}) $, all of our theorems still hold.



This claim is based on the following important observation. All proofs of our paper only relies on the local law (as well as its consequences), and therefore it suffices to show that such a local law is still valid for a general $ M \times N $ matrix. More specifically, the main task is to show that the optimal local semicircle law \eqref{eq:Local_SC_Law} still holds for the Girko symmetrization of an $ M \times N $ random matrix. Modulo the optimal local law, the optimal rigidity will still be valid as a by-product via standard approaches in random matrix theory.

For an $ M \times N $ matrix $ H $, we consider the augmented matrix $ H_\sfA $, which is an $ M \times M $ matrix by adding $ M-N $ columns to $ H $ with independent entries satisfying \eqref{e.Assumption1} and \eqref{e.Assumption2}. Without loss of generality, we may assume that the added columns are the first $ M-N $ columns in $ H_\sfA $. Since $ H_\sfA $ is a square matrix, the local semicircle law (see \cite[Theorem 1.1]{erdos2014imprimitive}) is still true. Specifically, for any fixed $ \omega>0 $, define the spectral domain
$$ \mathbf{S}=\mathbf{S}_\omega := \sth{z=x+\i y: |x| \leq \omega^{-1},M^{-1+\omega} \leq y \leq \omega^{-1}}. $$
Then for any $ z=x+\i y \in \mathbf{S} $, define the resolvent $ G^\sfA(z) := (\tH_\sfA - z)^{-1} $ and we have
$$ \abs{ \frac{1}{2M} \Tr G^\sfA(z) - m_\sc(z) } \prec \frac{1}{M y},  $$
and
$$ \max_{i,j}\abs{ G_{ij}^\sfA(z) - \delta_{ij} m_\sc(z) } \prec \pth{\frac{1}{My} + \sqrt{\frac{\im m_{sc}(z)}{My}}}. $$

For an $ n \times n $ matrix $ X $ and a subset $ 1 \leq k \leq n $, we define $ X^{(k)} $ as the $ (n -k ) \times (n -k) $ matrix
$$ X^{(k)} = (X_{ij})_{k+1 \leq i,j \leq n}. $$
Recall the Girko symmetrization $ \tH $ of a matrix $ H $. Then we have $ \tH = \tH_\sfA^{(M-N)} $.

For $ z=x+\i y $ with $ y>0 $, let $ G^{(k)}(z) := (\tH_\sfA^{(k)} - z)^{-1} $. In particular we have $ G^{(0)}(z) = G^\sfA(z) $ and $ G^{(M-N)}(z) = G(z) := (\tH -z)^{-1} $. Then we have the following resolvent identity (see e.g. \cite[Lemma 3.5]{benaych2017lectures})
$$ G_{ij}^\sfA(z) = G_{ij}^{(1)}(z) + \frac{G_{i1}^\sfA(z) G_{1j}^\sfA(z)}{G_{11}^\sfA(z)},\ \ \mbox{for all } 2 \leq i,j \leq 2M. $$
From the local semicircle law for the square matrix $ H_\sfA $, with overwhelming probability we have
$$ |G_{i,1}^\sfA(z)|, |G_{1,j}^\sfA(z)| \leq M^{\eps} \pth{\frac{1}{My} + \sqrt{\frac{\im m_{sc}(z)}{My}}},\ \ \mbox{and }\  |G_{1,1}^\sfA(z)| \sim 1.  $$
Using local law for $ G_{ij}^\sfA $ again, this implies that
$$  |G_{ij}^{(1)}(z) - G_{ij}^\sfA(z)| \lesssim  \frac{M^{2\eps}}{My}. $$
More generally, we have
$$ G_{ij}^{(k)}(z) = G_{ij}^{(k+1)}(z) + \frac{G_{i,k+1}^{(k)}(z) G_{k+1,j}^{(k)}(z)}{G_{k+1,k+1}^{(k)}(z)},\ \ \mbox{for all } k+2 \leq i,j \leq 2M. $$
By the same arguments as above, for any fixed $ k $, we can derive that
$$ |G_{ij}^{(k+1)}(z) - G_{ij}^{(k)}(z)| \lesssim  \frac{M^{2\eps}}{My}. $$
By a telescoping summation, we derive
\begin{align*}
|G_{ij}(z) - \delta_{ij} m_\sc(z)| &\leq \sum_{k=0}^{M-N-1} |G_{ij}^{(k+1)}(z) - G_{ij}^{(k)}(z)| + |G_{ij}^\sfA(z) - \delta_{ij} m_{\sc}(z)|\\
& \leq M^{\eps} \pth{\frac{1}{My} + \sqrt{\frac{\im m_{sc}(z)}{My}}} + (M-N) \frac{M^{2\eps}}{My}
\end{align*}
Since $ M-N=N^{o(1)} $, the second term can be absorb into the first term with a larger factor $ N^{3\eps} $. Thanks to the arbitrariness of $ \eps $, we have
\begin{equation}\label{eq:LSC_General_1}
|G_{ij}(z) - \delta_{ij} m_\sc(z)| \prec \pth{\frac{1}{Ny} + \sqrt{\frac{\im m_{sc}(z)}{Ny}}}.
\end{equation}

Moreover, the resolvent identity also yields
$$ \sum_{i=k+2}^{2M} G_{ii}^{(k)} (z) = \Tr G^{(k+1)}(z) + \frac{1}{G_{k+1,k+1}^{(k)}(z)} \sum_{i=k+2}^{2M} G_{i,k+1}^{(k)}(z) G_{k+1,i}^{(k)}(z).   $$
This yields
$$ \Tr G^{(k)}(z) - \Tr G^{(k+1)}(z) = \frac{1}{G_{k+1,k+1}^{(k)}(z)} \sum_{i=k+1}^{2M} G_{i,k+1}^{(k)}(z) G_{k+1,i}^{(k)}(z). $$
Using the Ward identity, this implies that
\begin{align*}
\abs{\Tr G^{(k)}(z) - \Tr G^{(k+1)}(z)} &\leq \frac{1}{|G_{k+1,k+1}^{(k)}(z)|} \sum_{i=k+1}^{2M} \abs{G_{i,k+1}^{(k)}(z)}^2\\
&\leq \frac{1}{|G_{k+1,k+1}^{(k)}(z)|} \frac{\im G_{k+1,k+1}^{(k)}(z)}{y} \leq \frac{1}{y}.
\end{align*}
From the local law for $ G^{\sfA} $, with overwhelming probability we have
$$ \abs{\frac{1}{2M} \Tr G^{\sfA}(z) -m_\sc(z)} \leq \frac{M^{\eps}}{My}  $$
Again, using $ M-N=N^{o(1)} $, the telescoping sum yields
$$ \abs{\frac{1}{M+N} \Tr G(z) - m_\sc(z)} \leq \frac{N^{\eps}}{Ny} + \frac{N^{o(1)}}{Ny}. $$
The second term can be absorbed into the first term for any fixed $ \eps>0 $. The arbitrariness of $ \eps $ concludes that
\begin{equation}\label{eq:LSC_General_2}
\abs{\frac{1}{M+N} \Tr G(z) - m_\sc(z)} \prec \frac{1}{Ny}.
\end{equation}

Hence, \eqref{eq:LSC_General_1} and \eqref{eq:LSC_General_2} have shown that the local law also holds for $ H $. The rigidity estimates Lemma \ref{l.Rigidity} also follow from a classical argument in random matrix theory (see e.g. \cite{benaych2017lectures}). As a consequence, Theorem \ref{thm:Smoothed_Singular}, Theorem \ref{t.Real}, Theorem \ref{thm:Smoothed_Condition} and Theorem \ref{thm:Condition} are all still valid.

\section*{Acknowledgment}
The author thanks Paul Bourgade for suggesting this problem and for insightful comments on an early version of this manuscript.

\appendix

\section{Auxiliary Results}

To make this paper self-contained, we collect some well-known results that are used in the paper.


The first result is about controlling the size of a martingale, which is used in Proposition \ref{p.aPriori} to bound the martingale term in the stochastic dynamics \eqref{eq:Dynamics_df}, and also in Lemma \ref{l.ImprovedLocalAverage} to bound \eqref{eq:Dynamics_dg}. This is from \cite[Appendix B.6, Equation (18)]{shorack2009empirical}.
\begin{lemma}\label{lem:Martingale}
For any continuous martingale $ M $ and any $ \lambda,\mu>0 $, we have
$$ \P \pth{\sup_{0 \leq u \leq t}|M_u| \geq \lambda, \langle M \rangle_t \leq \mu} \leq 2 e^{-\frac{\lambda^2}{2\mu}}. $$
\end{lemma}

The second result is the Helffer-Sj\H{o}strand formula, which is a classical result in functional calculus. This formula is used in Proposition \ref{prop:Resolvent_Comparison} to compute the trace of functions via the Stieltjes transform. We are using the version in \cite[Section 11.2]{erdos2017book}.
\begin{lemma}[Helffer-Sj\H{o}strand formula]\label{lem:HS_Formula}
Let $ f \in C^1(\R) $ with compact support and let $ \chi(y) $ be a smooth cutoff function with support in $ [-1,1] $, with $ \chi(y) =1 $ for $ |y| \leq \frac{1}{2} $ and with bounded derivatives. Then
$$ f(\lambda) = \frac{1}{\pi} \int_{\R^2} \frac{\i y f''(x) \chi(y) + \i (f(x) + \i f'(x)) \chi'(y)}{\lambda -x -\i y} \d x \d y. $$
\end{lemma}

We also have the following resolvent expansion identity. This is a well-known result in linear algebra, and it is used in Proposition \ref{prop:Resolvent_Comparison} to compare the resolvents of two matrices.

\begin{lemma}[Resolvent expansion]\label{lem:Resolvent_Expansion}
For any two matrices $ A $ and $ B $, we have
$$ (A+B)^{-1} = A^{-1} - (A+B)^{-1} B A^{-1} $$
provided that all the matrix inverses exist.
\end{lemma}

Finally, we have some estimates for the Stieltjes transform of the semicircle law. For $ z=E+\i \eta $ with $ \eta>0 $, recall that $ m_\sc(z) $ denotes the Stieltjes transform of the semicircle distribution. The following estimates are well known in random matrix theory (see e.g. \cite[Lemma 6.2]{erdos2017book}).

\begin{lemma}\label{lem:m_sc}
We have for all $ z=E+\i\eta $ with $ \eta $ that
$$ |m_\sc(z)| = |m_\sc(z) + z|^{-1} \leq 1. $$
Furthermore, there is a constant $ c>0 $ such that for $ E \in [-10,10] $ and $ \eta \in (0,10] $ we have
$$ c \leq |m_\sc(z)| \leq 1-c\eta,\ \ |1-m_\sc^2(z)| \sim \sqrt{\xi(E) + \eta}, $$
as well as
\begin{equation*}
\im m_\sc(z) \sim 
\begin{cases}
\sqrt{\xi(E) + \eta} & \mbox{if } |E| \leq 2,\\
\frac{\eta}{\sqrt{\xi(E)+\eta}} & \mbox{if } |E| \geq 2,
\end{cases}
\end{equation*}
where $ \xi(E) := ||E|-2| $ is the distance of $ E $ to the spectral edge.
\end{lemma}

\section{Proofs for Smoothed Analysis}
\subsection{Proof of Proposition \ref{p.aPriori}}\label{app:Apriori}

The proof is essentially the same as \cite[Proposition 3.8]{wang2019quantitative}, and we briefly describe the key steps here for completeness. For any $ 1 \leq \ell,m \leq N^{10} $, we define $ t_\ell=\ell N^{-10} $ and $ z^{(m)}=E_m+\i\eta_m=E_m+\i N^{-1+4\eps} \xi(E_m)^{-1/2} $, where $ \int_{-\infty}^{E_m} \d \rho_{\sc}=m N^{-10} $. Consider the stopping times
\begin{align*}
\tau_0 &=\inf\left\{ 0 \leq t \leq 1: \exists -N \leq k \leq N\ s.t.\  |s_k(t)-\gamma_k|> N^{-\frac{2}{3}+\eps}(N+1-|k|)^{-\frac{1}{3}} \right\},\\
\tau_{\ell,m} &=\inf\left\{ 0 \leq s \leq t_\ell : \im \tilde{\frakS}_s\left( z_{t_\ell - s}^{(m)} \right) >\dfrac{N^{2\eps}}{2}\dfrac{\xi(E_m)^{1/2}}{\left( \xi(E_m)^{1/2} \vee t_\ell \right)} \right\},\\
\tau &=\min\left\{ \tau_0,\tau_{\ell,m} : 0 \leq \ell,m\leq N^{10},\xi(E_m) >  N^{-\frac{2}{3}+4\eps} \right\},
\end{align*}
with the convention $ \inf \emptyset =\infty $. We claim that it suffices to show that $ \tau=\infty $ with overwhelming probability.

To prove this claim, for any $ z \in \S_\eps $ and $ 0 \leq t \leq 1 $, we pick $ t_\ell \leq t \leq t_{\ell+1} $ and $ |z-z^{(m)}| < N^{-5} $. Note that the maximum principle Lemma \ref{l.MaxPrin} implies $ |\psi_k(t)| \lesssim 1 $ for all $ k $ and $ t $. Then we have $ |\tilde{\frakS}_t(z) - \tilde{\frakS}_t(z^{(m)})| \lesssim N^{-2} $. Also, note that for $ z=E+i \eta $ we have $ |S_t(z)| \leq \eta^{-1} $, and
$$ \abs{\partial_z \tilde{\frakS}_t(z)} \lesssim N \max_k |\psi_k(0)| \eta^{-2} \lesssim N \eta^{-2},\ \ \abs{\partial_{zz} \tilde{\frakS}_t(z)} \lesssim N \eta^{-3}. $$
Consider the events
$$ \calE_{\ell,m,k} := \sth{ \sup_{t_\ell \leq u \leq t_{\ell+1}} \abs{\int_{t_\ell}^u \frac{e^{-\frac{v}{2}} \psi_k(v) } {(s_k(v) - z^{(m)})^2} \d B_k(v) } < N^{-4} }. $$
On the event $ \bigcap_k \calE_{\ell,m,k} $, the above estimates imply that $ |\tilde{\frakS}_t(z^{(m)}) - \tilde{\frakS}_{t_\ell}(z^{(m)})| < N^{-2} $. It further shows that
$$ \abs{\tilde{\frakS}_t(z) - \tilde{\frakS}_{t_\ell}(z^{(m)}) } < N^{-2}. $$
Since this holds for all $ z $ and $ t $, we have shown that
\begin{equation}\label{eq:Set_Inclusion}
\sth{\tau=\infty} \bigcap \pth{ \bigcap_{1 \leq \ell,m \leq N^{10},-N \leq k \leq N} \calE_{\ell,m,k} }  \subset \bigcap_{z \in \S_\eps, 0\leq t \leq 1} \sth{ \im \tilde{\frakS}_t(z) \leq N^{2\eps} \frac{\xi(E)^{1/2}}{ \pth{\xi(E)^{1/2} \vee t}} } . 
\end{equation}
Moreover, note that
\begin{multline*}
\left\langle \int_{t_\ell}^{t_{\ell+1}} \frac{e^{-\frac{v}{2}} \psi_k(v) } {(s_k(v) - z^{(m)})^2} \d B_k(v) \right\rangle_{t_{\ell+1}}\\
\leq \int_{t_\ell}^{t_{\ell+1}} \frac{e^{-v} |\psi_k(v)|^2}{(s_k(v) - z^{(m)})^4} \d v \leq N^{-10} \pth{N^{-1+4\eps}}^{-4} \pth{\max_k |\psi_k(0)|}^2 \leq N^{-6+16\eps}.
\end{multline*}
Using Lemma \ref{lem:Martingale}, we conclude that the event $ \calE_{\ell,m,k} $ happens with overwhelming probability. By a union bound, we further have that $ \bigcap_{l,m,k} \calE_{\ell,m,k} $ happens with overwhelming probability. Together with the set inclusion \eqref{eq:Set_Inclusion}, we conclude that the claim is true, i.e. it suffices to prove $ \tau=\infty $ with overwhelming probability.

To prove $ \tau=\infty $ with overwhelming probability, consider some fixed $ t=t_\ell $ and $ z=z^{(m)} = E+\i\eta $, and define the function $ f_u(z):=\tilde{\frakS}_u(z_{t-u}) $. By Lemma \ref{l.InitialCondition}, the initial condition is well controlled $ \im \tilde{\frakS}_0(z) \lesssim N^{2\eps} \tfrac{\xi(E_m)^{1/2}}{(\xi(E_m)^{1/2} \vee t)} $. To bound the increments, note that the dynamics \eqref{e.AdvectionEqn} yields
\begin{equation}\label{eq:Dynamics_df}
\d f_{u \wedge \tau}(z) = \ep_u(z_{t-u})\d (u \wedge \tau) - \dfrac{e^{-\frac{u}{2}}}{\sqrt{N}}\sum_{-N \leq k \leq N}\dfrac{\psi_k(u)}{(z_{t-u}-s_k(u))^2}\d B_k(u \wedge \tau),
\end{equation}
where
\begin{equation*}
\ep_u(z) := (S_u(z)-m_{\sc}(z))\partial_z \tilde{\frakS}_u + \dfrac{1}{4N} (\partial_{zz} \tilde{\frakS}_u) + \dfrac{e^{-\frac{u}{2}}}{2N} \sum_{-N \leq k \leq N} \dfrac{\psi_k(u)}{(s_k -z)^2 (s_k+z)}
\end{equation*}
By the local semicircle law \eqref{eq:Local_SC_Law}, we have
\begin{equation}\label{eq:Modify.Apriori_SCL}
\begin{aligned}
&\quad \sup_{0 \leq s \leq t} \left| \int_0^s  \left( S_u(z_{t-u}) - m_{\sc}(z_{t-u}) \right) \partial_z \tilde{\frakS}_u(z_{t-u}) \d (u \wedge \tau) \right|\\
&\leq  \int_0^t \dfrac{N^\eps}{{N \im (z_{t-u})}}  \sum_{-N \leq k \leq N} \dfrac{\psi_k(u)}{|s_k(u) - z_{t-u}|^2} \d (u \wedge \tau)\\
&\leq  \int_0^t \dfrac{N^\eps \im \tilde{\frakS}_u(z_{t-u})}{{N} (\im z_{t-u})^{2}} \d(u \wedge \tau)\\
& \leq \int_0^t \dfrac{N^{2\eps} \d u}{{N}\left( \eta + (t-u)\frac{b(z)}{\xi(z)^{1/2}} \right)^{2}} \dfrac{\xi(E)^{\frac{1}{2}}}{(\xi(E)^{\frac{1}{2}} \vee t)}\\
& \lesssim \dfrac{\xi(E)^{\frac{1}{2}}}{(\xi(E)^{\frac{1}{2}} \vee t)}.
\end{aligned}
\end{equation}
Also, we have
$$ \sup_{0 \leq s \leq t} \left| \int_0^s \dfrac{1}{4N}(\partial_{zz} \tilde{\frakS}_u(z_{t-u})) \d (u \wedge \tau) \right| \lesssim \int_0^t \dfrac{\im \tilde{\frakS}_u(z_{t-u})}{N \left( \im (z_{t-u}) \right)^2}\d (u \wedge \tau)  \lesssim N^{-\eps} \dfrac{\xi(E)^{\frac{1}{2}}}{ (\xi(E)^{\frac{1}{2}} \vee t)}, $$
and
\begin{multline*}
\sup_{0 \leq s \leq t} \left| \int_0^s \dfrac{e^{-\frac{u}{2}}}{2N} \sum_{-N \leq k \leq N} \dfrac{\psi_k(u)}{(s_k - z_{t-u})^2(s_k+z_{t-u})} \d (u \wedge \tau) \right|\\
\lesssim \dfrac{1}{N} \int_0^t \dfrac{1}{\im (z_{t-u})} \sum_{-N \leq k \leq N}\dfrac{\psi_k(u)}{|s_k(u) - z_{t-u}|^2} \d (u \wedge \tau)\\
\lesssim \int_0^t  \dfrac{\im f_u(z_{t-u})}{N \left( \im (z_{t-u}) \right)^2} \d (u \wedge \tau) \lesssim N^{-\eps} \dfrac{\xi(E)^{\frac{1}{2}}}{ (\xi(E)^{\frac{1}{2}} \vee t)}.
\end{multline*}
For the martingale part
$$ M_s := \int_0^s \dfrac{e^{-\frac{u}{2}}}{\sqrt{N}} \sum_{-N \leq k \leq N} \dfrac{\psi_k(u)}{(z_{t-u}-s_k(u))^2}\d B_k(u \wedge \tau), $$
using the rigidity of singular values (Lemma \ref{l.Rigidity}), with overwhelming probability we have
$$ \sup_{0 \leq s \leq t} |M_s|^2 \lesssim N^{\frac{\eps}{2}} \int_0^t \frac{1}{N} \sum_{-N \leq k \leq N} \frac{|\psi_k(u)|^2}{|z_{t-u} - \gamma_k|^4} \d (u \wedge \tau). $$
To estimate this integral, we chop the interval $ [-N,N] $ into $ 2N^{1-4\eps} $ subintervals $ I_j=\llbracket k_j,k_{j+1} \rrbracket $ where $ k_j = -N+\lfloor j N^{4\eps} \rfloor $. We can bound the summation in the integral in the following way
$$ \frac{1}{N} \sum_{-N \leq k \leq N} \frac{|\psi_k(u)|^2}{|z_{t-u} - \gamma_k|^4} \leq \frac{1}{N} \sum_{0 \leq j \leq 2N^{1-4\eps}} \pth{\max_{k \in I_j} \psi_k(u)} \pth{\max_{k \in I_j} \frac{1}{|z_{t-u}-\gamma_k|^4}} \pth{\sum_{k \in I_j}\psi_k(u)}. $$
Using similar discretization arguments as above, we can derive
$$ \max_{k \in I_j} \psi_k(u) \leq \sum_{k \in I_j} \psi_k(u) \leq \frac{N^{6\eps}}{ N (\xi(\gamma_{k_j})^{1/2} \vee u)}, $$
and
$$ \max_{k \in I_j} \frac{1}{|z_{t-u}-\gamma_k|^4} \leq N^{-4\eps} \sum_{k \in I_j} \frac{1}{|z_{t-u}-\gamma_k|^4}. $$
Therefore, we obtain
$$ \sup_{0 \leq s \leq t} |M_s|^2 \leq N^{-2+9 \eps} \int_0^t \int_{-2}^2 \frac{1}{|z_{t-u}-x|^4 (\xi(x) \vee u^2) } d \rho_\sc(x) \d u \lesssim N^{\eps} \frac{\xi(E)}{(\xi(E) \vee t^2)}. $$


Combining this estimate for the martingale term with previous estimates, a union bound shows that with overwhelming probability we have
$$ \sup_{\ell,m,0\leq s \leq t_\ell,\xi(E_m)>\varphi^2 N^{-2/3}} \im \tilde{f}_{s \wedge \tau}(z_{t_\ell - s \wedge \tau}^{(m)}) \lesssim N^{2 \eps} \dfrac{\xi(E)^{1/2}}{\left( \xi(E)^{1/2} \vee t \right)}.  $$
Now we have proved $ \tau=\infty $ with overwhelming probability and hence the desired result is true.

\subsection{Proof of Lemma \ref{l.Bootstrap}}\label{app:Bootstrap}

The proof of Lemma \ref{l.Bootstrap} is a delicate task. The key part of the proof is a careful analysis of the dynamics. The main idea is to approximate the dynamics with a short-range version, which will be easier to control. To do this, we show the finite speed of propagation estimate for the short-range kernel of the parabolic-type equation \eqref{e.Parabolic} satisfied by $ \{\varphi_k\} $. Then we prove a short-range approximation of the original dynamics and introduce a regularized equation. Finally, we show that, with a well-behaved initial condition, the regularized equation gives us the desired good approximation.

To begin with, the core input of the bootstrap argument is the following technical lemma, which states that the estimate of the local average will improve along with the induction hypothesis $ \calH_{\alpha} $.

\begin{lemma}\label{l.ImprovedLocalAverage}
Assume $ \calH_{\alpha} $. Let $ \xi>0 $ be any fixed small constant. For any $ 0<t<1 $, any $ \eps_0>0 $ arbitrarily small and $ z=E+\i\eta $ satisfying $  N^{-1+\eps_0}<\eta<1 $, $ |E|<2-\xi $, we have
\begin{equation}\label{e.ImprovedLocalAverage}
\left| \im \frakS_t(z) - e^{-\frac{t}{2}}\dfrac{\im S_t(z)}{\im S_0(z_t)} \im \frakS_0(z_t) \right| \prec \left( \dfrac{(Nt)^\alpha}{N^2 t \eta} + \dfrac{1}{Nt} \right)
\end{equation}
\end{lemma}
\begin{proof}
Fix $ t $ and consider the function
$$ g_u(z) := \frakS_u(z_{t-u})-e^{-\frac{u}{2}}\dfrac{\im \frakS_0(z_t)}{\im S_0(z_t)} S_u(z_{t-u}),\ \ 0 \leq u \leq t. $$
An observation is that $ e^{-u/2}S_u(z) $ satisfy the same stochastic advection equation \eqref{e.AdvectionEqn} with $ \varphi_k $ replaced by $ \tfrac{1}{2N} $. Therefore, we have
\begin{equation}\label{eq:Dynamics_dg}
\begin{aligned}
\d g_u &= \left( S_u(z_{t-u}) - m_{\sc}(z_{t-u}) \right)  \left( \partial_z \frakS_u(z_{t-u}) - e^{-\frac{u}{2}}\dfrac{\im \frakS_0(z_t)}{\im S_0(z_t)} \partial_z S_u(z_{t-u}) \right) \d u\\
&\quad + \dfrac{1}{4N}\left( \partial_{zz} \frakS_u(z_{t-u}) - e^{-\frac{u}{2}}\dfrac{\im \frakS_0(z_t)}{\im S_0(z_t)} \partial_{zz} S_u(z_{t-u}) \right)\d u\\
&\quad + \dfrac{e^{-\frac{u}{2}}}{2N} \sum_{-N \leq k \leq N} \dfrac{\theta_k(u)}{(s_k(u)-z_{t-u})^2 (s_k(u) + z_{t-u})} \d u\\
&\quad - \dfrac{e^{-\frac{u}{2}}}{\sqrt{N}} \sum_{-N \leq k \leq N} \dfrac{\theta_k(u)}{(s_k(u) - z_{t-u})^2} \d B_k,
\end{aligned}
\end{equation}
where
$$ \theta_k(u) = \varphi_k(u) - \frac{\im \frakS_0(z_t)}{2N \im S_0(z_t)}. $$
Similarly as in the proof of Proposition \ref{p.aPriori}, for $ \eps>0 $ and $ 0\leq \ell,m,p \leq N^{10} $, define $ t_\ell=\ell N^{-10} $ and $ z^{(m,p)} = E_m + \i \eta_p $ where $ \int_{-\infty}^{E_m} \d \rho = mN^{-10} $ and $ \eta_p = N^{-1+4\eps}+pN^{-10} $. We also pick $ c>0 $ such that $ \lfloor (1-c) N \rfloor =\arg\min_{k} |\gamma_k - (2-\tfrac{\xi}{10})| $. Assuming $ \calH_\alpha $, let $ \eps_0>0 $ be the arbitrarily small scale in the hypothesis. Let $ C>0 $ be some suitably large constant. Recall the stopping times
\begin{align*}
\tau_0 &=\inf\left\{ 0 \leq u \leq 1: \exists -N \leq k \leq N\ s.t.\  |s_k(u)-\gamma_k|>N^{-\frac{2}{3}+\eps}(N+1-|k|)^{-\frac{1}{3}} \right\},\\
\tau_1 &=\inf\left\{ N^{-1+\eps_0} \leq u \leq 1: \exists -N \leq k \leq N\ s.t.\ |\varphi_k(u)| > \dfrac{N^{C\eps}}{N}\dfrac{1}{\left((\frac{N+1-|k|}{N})^{1/3} \vee u\right)} \right\},
\end{align*}
and consider the new stopping times
\begin{align*}
\tau_2 &=\inf \left\{ N^{-1+\eps_0} \leq u \leq 1 : \exists k \in \llbracket (c-1)N,(1-c)N \rrbracket\ s.t.\ |\varphi_k(u) - \uu_k(u)| > N^{C\eps} \dfrac{(Nt)^\alpha}{N^2 t} \right\},\\
\tau_{\ell,m,p} &=\inf\left\{0 \leq u \leq t_\ell : \left| \im g_u^{(t_\ell)}(z^{(m,p)}) \right|> N^{C\eps} \left( \dfrac{(Nt)^\alpha}{N^2 t \eta} + \dfrac{1}{Nt} \right) \right\}\\
\tau &= \min\{\tau_0,\tau_1,\tau_2,\tau_{\ell,m,p} : 0 \leq \ell,m,p \leq N^{10},|E_m|<2-\xi\}.
\end{align*}
Recall the convention $ \inf \emptyset =\infty $. As shown in the proof of Proposition \ref{p.aPriori}, it suffices to show that $ \tau=\infty $ with overwhelming probability.

A key ingredient for the analysis of the dynamics of $ g_u $ is the following estimates on $ \theta_k(u) $. To do this, we fix some $ t=t_\ell $ and $ z=z^{(m,p)} $ with $ |E_m|<2-\xi $, and let $ N^{-1+\eps_0} \leq u \leq t \wedge \tau $ and $ k \in \llbracket (c-1)N,(1-c)N \rrbracket $.

On the one hand, we have a direct a priori estimate. Since $ u \leq \tau_1 $, we have
$$ |\varphi_k(u)| \lesssim  N^{-\frac{2}{3}+C\eps}(N+1-|k|)^{-\frac{1}{3}}. $$
Moreover, note that for $ z=E+i\eta $ with $ E $ in the bulk and $ t<1 $, uniformly we have $ \im S_0(z_t) \gtrsim 1 $. By Lemma \ref{l.InitialCondition}, this shows
$$ \dfrac{\im \frakS_0(z_t)}{2N \im S_0(z_t)} \lesssim N^{-1+\eps} \lesssim N^{-\frac{2}{3}+C\eps}(N+1-|k|)^{-\frac{1}{3}}. $$
As a consequence, we have
\begin{equation}\label{e.AprioriDelta}
|\theta_k(u)| \lesssim  N^{-\frac{2}{3}+C\eps}(N+1-|k|)^{-\frac{1}{3}}.
\end{equation}

On the other hand, the estimate can also be obtained via approximation
$$ |\theta_k(u)| \leq |\varphi_k(u) - \uu_k(u)|+|\uu_k(u) - \uu_j(t)|+\left| \uu_j(t) - \frac{\im f_0(z_t)}{2N \im S_0(z_t)} \right|. $$
For the first term, since $ u \leq \tau_2 $ we have
$$ |\varphi_k(u) - \uu_k(u)| \leq N^{C\eps} \dfrac{(Nu)^a}{N^2 u}. $$
Choosing $ |\gamma_j -E|\leq N^{-1+2\eps} $, the remaining two terms are controlled by Lemma \ref{l.SizeUbar}
\begin{align*}
|\uu_k(u) - \uu_j(t)|+\left| \uu_j(t) - \frac{\im f_0(z_t)}{2N \im S_0(z_t)} \right| &\lesssim N^\eps \left( \dfrac{|k-j|}{N^2 u}+\dfrac{|t-u|}{Nu} + \dfrac{|E-\gamma_j|}{Nt} + \dfrac{\eta}{Nt} \right)\\
& \lesssim N^{C\eps} \left( \dfrac{|\gamma_k -E|}{Nu}+\dfrac{t-u}{Nu}+\dfrac{\eta}{Nt} \right).
\end{align*}
Together, we decompose the error terms into two parts and obtain the following
\begin{equation}\label{e.EstimateDelta}
|\theta_k(u)| \leq N^{C\eps} \left( \dfrac{|\gamma_k -E|}{Nu}+\dfrac{t-u}{Nu}+\dfrac{\eta}{Nt} +\dfrac{(Nu)^\alpha}{N^2 u} \right) =: \varphi^C \left( \dfrac{|\gamma_k -E|}{Nu} + \La(a,N,t,\eta,u) \right)
\end{equation}

With the above control on $ \theta_k(u) $, the dynamics \eqref{eq:Dynamics_dg} can be used to bound $ \im(g_t-g_0) $ similarly as in Proposition \ref{p.aPriori}. For the first term, we have
\begin{equation}\label{eq:Modify.LocalAverage_SCL}
\begin{aligned}
&\quad \int_0^{t\wedge\tau} |S_u(z_{t-u})-m_{\sc}(z_{t-u})|\left| \partial_z \frakS_u(z_{t-u}) - e^{-\frac{u}{2}}\dfrac{\im \frakS_0(z_t)}{\im S_0(z_t)} \partial_z S_u(z_{t-u}) \right| \d u\\
&\leq \int_0^{t\wedge\tau} \dfrac{N^{C\eps}}{{N \im (z_{t-u})}}  \sum_{-N \leq k \leq N} \dfrac{|\theta_k(u)|}{|s_k(u)-z_{t-u}|^2} \d u\\
&\leq \int_0^{t\wedge\tau} \dfrac{N^{C\eps}}{{N \im (z_{t-u})}} \left( \sum_{|k| \geq (1-c)N}\dfrac{|\theta_k(u)|}{|\gamma_k-z_{t-u}|^2} + \sum_{|k| < (1-c)N}\dfrac{|\theta_k(u)|}{|\gamma_k-z_{t-u}|^2} \right) \d u\\
&=: I_1 + I_2.
\end{aligned}
\end{equation}
For the soft edge part $ |k| \geq (1-c)N $, using \eqref{e.AprioriDelta} we obtain
\begin{equation}\label{eq:Modify.LocalAverage_SoftEdge}
\begin{aligned}
I_1 &\leq N^{C\eps} \int_0^{t \wedge \tau}\dfrac{1}{{N \im (z_{t-u})}} \sum_{|k| \geq (1-\alpha)N}  N^{-\frac{2}{3}+C\eps}(N+1-|k|)^{-\frac{1}{3}} \d u\\
& \leq  N^{C\eps} \int_0^{t \wedge \tau} \dfrac{1}{{N \im (z_{t-u})}} \d u\\
& \leq N^{C\eps} \dfrac{\log(1+Nt)}{N}.
\end{aligned}
\end{equation}
For $ I_2 $, note that
\begin{align*}
\sum_{|k| < (1-c)N}\dfrac{|\theta_k(u)|}{|\gamma_k - z_{t-u}|^2} &\leq N^{C\eps} \left( \dfrac{1}{Nu} \sum_{|k| < (1-c)N}\dfrac{|\gamma_k-E|}{|\gamma_k - z_{t-u}|^2} + \La \sum_{|k| < (1-c)N}\dfrac{1}{|\gamma_k - z_{t-u}|^2} \right) \\
&\leq N^{C\eps} \left( \dfrac{1}{u}+N\dfrac{\La}{\eta+t-u} \right)
\end{align*}
This yields
\begin{equation}\label{eq:Modify.LocalAverage_Bulk_LongTime}
\begin{aligned}
&\quad \int_{N^{-1+\eps_0}}^{t \wedge \tau} \dfrac{N^{C\eps}}{{N \im (z_{t-u})}} \sum_{|k| < (1-\alpha)N}\dfrac{|\theta_k(u)|}{|\gamma_k - z_{t-u}|^2} \d u\\
& \leq N^{C\eps} \int_{N^{-1+\eps_0}}^{t \wedge \tau} \dfrac{1}{{N (\eta+t-u)}} \left( \dfrac{1}{u}+N\dfrac{\La}{\eta+t-u} \right) \d u\\
& \leq N^{C\eps} \int_{N^{-1+\eps_0}}^{t \wedge \tau} \dfrac{1}{{N (\eta+t-u)}} \left( \dfrac{1}{u}+N\dfrac{1}{\eta+t-u}\left( \dfrac{t-u}{Nu}+\dfrac{\eta}{Nt} +\dfrac{(Nu)^\alpha}{N^2 u} \right) \right) \d u\\
& \leq N^{C\eps} \left( \dfrac{(Nt)^\alpha}{N^2 t\eta} + \dfrac{1}{Nt} \right)
\end{aligned}
\end{equation}
Moreover, without loss of generality, we may assume that $ \eps_0 \sim \eps $. Then, using \eqref{e.AprioriDelta} we obtain 
\begin{equation}\label{eq:Mofity.LocalAverage_Bulk_ShortTime}
\int_0^{N^{-1+\eps_0}} \dfrac{N^{C\eps}}{{N \im (z_{t-u})}} \sum_{|k| < (1-\alpha)N}\dfrac{|\theta_k(u)|}{|\gamma_k - z_{t-u}|^2} \d u \leq \dfrac{N^{C\eps}}{N^{2}(\eta+t)^{2}}
\end{equation}
Together with previous estimates, this shows
\begin{equation}\label{eq:Mofity.LocalAverage_SCL_Bound}
\int_0^{t\wedge\tau} |S_u(z_{t-u})-m_{\sc}(z_{t-u})|\left| \partial_z \frakS_u(z_{t-u}) - e^{-\frac{u}{2}}\dfrac{\im \frakS_0(z_t)}{\im S_0(z_t)} \partial_z S_u(z_{t-u}) \right| \d u
\leq N^{C\eps} \left( \dfrac{(Nt)^\alpha}{N^2 t\eta} + \dfrac{1}{Nt} \right).
\end{equation}
Similarly, we have
$$ \int_0^{t \wedge \tau}\dfrac{1}{4N}\left| \partial_{zz} \frakS_u(z_{t-u}) - e^{-\frac{u}{2}}\dfrac{\im \frakS_0(z_t)}{\im S_0(z_t)} \partial_{zz} S_u(z_{t-u})\right| \d u \leq N^{C\eps} \left( \dfrac{(Nt)^\alpha}{N^2 t \eta}+\dfrac{1}{Nt} \right) $$
and
$$ \int_0^{t \wedge \tau} \left|\dfrac{e^{-\frac{u}{2}}}{2N} \sum_{-N \leq k \leq N} \dfrac{\theta_k(u)}{(s_k(u)-z_{t-u})^2 (s_k(u) + z_{t-u})}\right| \d u \leq N^{C\eps} \left( \dfrac{(Nt)^\alpha}{N^2 t \eta}+\dfrac{1}{Nt} \right) $$

It suffices to bound the martingale term
$$ M_s := \int_0^{s} \dfrac{e^{-\frac{u}{2}}}{\sqrt{N}} \sum_{-N \leq k \leq N} \dfrac{\theta_k(u)}{(s_k(u) - z_{t-u})^2} \d B_k .$$
Again we decompose the integral into two parts
\begin{align*}
\left\langle M \right\rangle_{t \wedge \tau} &\lesssim \dfrac{1}{N}\int_0^{t \wedge \tau} \sum_{|k| \geq (1-c)N}\dfrac{|\theta_k(u)|^2}{|\gamma_k-z_{t-u}|^4} \d u + \dfrac{1}{N}\int_0^{t \wedge \tau} \sum_{|k| < (1-c)N}\dfrac{|\theta_k(u)|^2}{|\gamma_k-z_{t-u}|^4} \d u\\
&=: J_1 + J_2.
\end{align*}
The contribution from the soft edge is easy to control
$$ J_1 \leq \dfrac{N^{C\eps}}{N}\int_0^{t \wedge \tau} \sum_{|k| \geq (1-c)N}\left( N^{-\frac{2}{3}} (N+1-|k|)^{-\frac{1}{3}} \right)^2 \d u \leq N^{C\eps} \dfrac{t}{N^2} $$
For the other term, we use both \eqref{e.AprioriDelta} and \eqref{e.EstimateDelta}
\begin{align*}
J_2 &\lesssim \dfrac{N^{C\eps}}{N} \int_0^{t \wedge \eta} \dfrac{1}{N^2}\sum_{|k| < (1-c)N} \dfrac{1}{|\gamma_k - z_{t-u}|^4} \d u\\
&\quad + \dfrac{N^{C\eps}}{N} \int_{t \wedge \eta}^{t \wedge \tau} \sum_{|k| < (1-c)N} \dfrac{1}{|\gamma_k - z_{t-u}|^4} \qth{ \left( \dfrac{|\gamma_k -E|^2}{N^2 u^2} + \Lambda^2 \right) \wedge \dfrac{1}{N^2} } \d u.
\end{align*}
Note that
$$ \dfrac{N^{C\eps}}{N} \int_0^{t \wedge \eta} \dfrac{1}{N^2}\sum_{|k| < (1-c)N} \dfrac{1}{|\gamma_k - z_{t-u}|^4} \d u \leq \dfrac{N^{C\eps}}{N} \int_0^{t \wedge \eta} \dfrac{1}{N(\eta+t-u)^3} \d u \leq \dfrac{N^{C\eps}}{N^2 t^2}. $$
For the other term, without loss of generality we may assume $ \eta<t $. For the $ \tfrac{|\gamma_k -E|^2}{N^2 u^2} $ term, we have
\begin{align*}
&\quad \dfrac{N^{C\eps}}{N} \int_{\eta}^{t} \sum_{|k| < (1-c)N} \dfrac{1}{|\gamma_k - z_{t-u}|^4}  \left( \dfrac{|\gamma_k -E|^2}{N^2 u^2} \wedge \dfrac{1}{N^2} \right) \d u\\
&\leq \dfrac{N^{C\eps}}{N^2 t^2} + \dfrac{N^{C\eps}}{N} \int_{\eta}^{t} \sum_{k: |\gamma_k-E| \leq u} \dfrac{1}{|\gamma_k - z_{t-u}|^4} \dfrac{|\gamma_k -E|^2}{N^2 u^2} \d u\\
&\leq \dfrac{N^{C\eps}}{N^2 t^2} + N^{C\eps} \int_{\eta}^{t} \dfrac{1}{N^2 u^2} \left( \int_{-u}^u \dfrac{x^2}{x^4+(\eta+t-u)^4} \d x \right) \d u\\
& \leq \dfrac{N^{C\eps}}{N^2 t^2}.
\end{align*}
For the contribution of $ \Lambda(a,N,t,\eta,u) $, we have
\begin{align*}
&\quad \dfrac{N^{C\eps}}{N} \int_{\eta}^{t} \sum_{|k| < (1-c)N} \dfrac{1}{|\gamma_k - z_{t-u}|^4} \left( \Lambda^2 \wedge \dfrac{1}{N^2} \right) \d u\\
&\leq \dfrac{N^{C\eps}}{N} \int_{\eta}^{t} \sum_{|k| < (1-c)N} \dfrac{1}{|\gamma_k - z_{t-u}|^4}  \left[ \left( \dfrac{(t-u)^2}{N^2 u^2} + \dfrac{(Nu)^{2\alpha}}{N^4 u^2} + \dfrac{\eta^2}{N^2 t^2} \right) \wedge \dfrac{1}{N^2} \right] \d u\\
& \leq N^{C\eps}\int_{\eta}^{t} \dfrac{1}{(\eta+t-u)^3}\left[ \dfrac{(Nu)^{2\alpha}}{N^4 u^2} + \dfrac{\eta^2}{N^2 t^2} + \left( \dfrac{(t-u)^2}{N^2 u^2} \wedge \dfrac{1}{N^2} \right) \right] \d u.
\end{align*}
The first two terms in the bracket give us
$$ \int_{\eta}^{t} \dfrac{1}{(\eta+t-u)^3} \left( \dfrac{(Nu)^{2\alpha}}{N^4 u^2} + \dfrac{\eta^2}{N^2 t^2} \right) \d u \leq \dfrac{(Nt)^{2\alpha}}{N^4 t^2 \eta^2}+\dfrac{1}{N^2 t^2}. $$
For the remaining term, we have
\begin{multline*}
\int_{\eta}^{t} \dfrac{1}{(\eta+t-u)^3} \left( \dfrac{(t-u)^2}{N^2 u^2} \wedge \dfrac{1}{N^2} \right) \d u\\
\leq \int_{\eta}^{\frac{t}{2}} \dfrac{1}{(\eta+t-u)^3} \dfrac{1}{N^2} \d u + \int_{\frac{t}{2}}^t \dfrac{1}{(\eta+t-u)^3} \dfrac{(t-u)^2}{N^2 u^2} \d u \leq \dfrac{\log N}{N^2 t^2}
\end{multline*}
Combining these results shows
$$ \left\langle M \right\rangle_{t \wedge \tau} \leq N^{C\eps} \left( \dfrac{(Nt)^{2\alpha}}{N^4 t^2 \eta^2}+\dfrac{1}{N^2 t^2} \right). $$
Using Lemma \ref{lem:Martingale} and a union bound, for any fixed large $ D>0 $ and sufficiently large $ N > N_0(\eps,D) $, we have
$$ \P\left( \sup_{\ell,m,p,0\leq s \leq t \wedge \tau,|E_m|<2-\xi} |M_s| \leq N^{C\eps} \left( \dfrac{(Nt)^{\alpha}}{N^2 t \eta}+\dfrac{1}{N t} \right) \right) > 1-N^{-D}. $$

Together with previous estimates, with overwhelming probability we have
$$  \sup_{\ell,m,p,0\leq s \leq t \wedge \tau,|E_m|<2-\xi} |g_s(z)| \leq N^{C\eps} \left( \dfrac{(Nt)^\alpha}{N^2 t \eta} + \dfrac{1}{Nt} \right).  $$
This implies $ \min_{\ell,m,p} \{\tau_{\ell,m,p}\}=\infty $ with overwhelming probability. Moreover, we have shown in Lemma \ref{l.Rigidity}, Lemma \ref{lem:phi_decay_rough} that $ \tau_0=\infty $ and $ \tau_1=\infty $ with overwhelming probability. Assuming the hypothesis $ \calH_\alpha $, we also have $ \tau_2=\infty $ with overwhelming probability. These imply that $ \tau=\infty $ with overwhelming probability. Hence we complete the proof.
\end{proof}

Now we move on to the short-range approximation of the dynamics. Recall that $ \{\varphi_k\} $ satisfies the parabolic equation \eqref{e.Parabolic}, and we rewrite it as
$$ \dfrac{\d}{\d t}\varphi_k =\left( \calP \varphi \right)_k $$
where the time-dependent operator $ \calP $ is defined in the following way: For $ f : \R \to \R^{2N} $,
$$ \left( \calP f \right)_k := \sum_{j \neq \pm k} c_{jk}(t)(f_j(t)-f_k(t)),\ \ \ c_{jk}(t):=\dfrac{1}{2N(s_j(t) - s_k(t))^2}. $$
Consider some parameter $ l=l(N,\alpha) $ which will be determined later, we decompose the operator $ \calP $ into two parts $ \calP=\calP_{\mathrm{short}}+\calP_{\mathrm{long}} $. The operators $ \calP_{\mathrm{short}} $ and $ \calP_{\mathrm{long}} $ represent the short-range interactions and long-range interactions respectively and are defined as follows
\begin{align*}
\left( \calP_{\mathrm{short}} f \right)_k &=\sum_{|j-k|\leq l} c_{jk}(t)(f_j(t) - f_k(t)),\\
\left( \calP_{\mathrm{long}} f \right)_k &=\sum_{|j-k|> l} c_{jk}(t)(f_j(t) - f_k(t)).
\end{align*}
Note that the operators $ \calP_{\mathrm{short}},\calP_{\mathrm{long}} $ are also time dependent. Let $ \Us(s,t) $ denote the semigroup associated with the operator $ \calP_{\mathrm{short}} $ in the sense
$$ \partial_t \Us(s,t)=\calP_{\mathrm{short}}(t)\Us(s,t),\ \ \Us(s,s)=\Id. $$
Also, let $ \sT $ denote the semigroup associated with $ \calP $.

To prove the short-range approximation, we need the following finite speed of propagation estimate for the semigroup. Such estimates were proved in \cite{che2019universality} with minor changes.

\begin{lemma}\label{l.FiniteSpeed}
For any fixed small $ c>0 $ and large $ D>0 $, there exists $ N_0(c,D) $ such that the following holds with probability at least $ 1-N^{-D} $. For any $ \eps>0 $, $ N>N_0 $, $ 0<u<v<1 $, $ l \geq N|u-v| $, $ |k| \leq (1-c)N $ and $ -N \leq j \leq N $ such that $ |k-j|>N^\eps l $, we have
\begin{equation}\label{e.FiniteSpeed}
\left( \Us (u,v) \de_k \right)(j) < N^{-D}.
\end{equation}
\end{lemma}

With Lemma \ref{l.FiniteSpeed}, we have the following short-range approximation estimate. In particular, this short-range approximation can be improved based on the hypothesis $ \calH_\alpha $.

\begin{lemma}\label{l.ShortRange}
Assume $ \calH_{\alpha} $. Let $ c>0 $ be any fixed small constant. There exists a constant $ C>0 $ such that for any $ \eps>0 $, $ N^{-1+C\eps} <t<1 $, $ \tfrac{t}{2} \leq u<v \leq t $, $ l \geq N^{\eps} $ and $ |k| \leq (1-c)N $, we have
\begin{equation}\label{e.ShortRange}
\left| \left(\left( \sT(u,v) - \Us(u,v) \right)\varphi(u)\right)_k \right| \prec (v-u) \left( \dfrac{N}{l}\dfrac{(Nt)^\alpha}{N^2 t} + \dfrac{1}{Nt} \right).
\end{equation}
\end{lemma}

\begin{proof}
The Duhamel's principle implies
$$ \pth{ \pth{ \sT(u,v)-\Us(u,v) } \varphi(u) }_k = \int_u^v \pth{ \Us(s,v) [ \pth{\calP_{\mathrm{long}} \, \varphi}(s) ] }_k \d s. $$
On the event that Lemma \ref{l.FiniteSpeed} holds, for $  |k| \leq (1-3c)N $, the finite speed of propagation yields
$$ \pth{ \Us(s,v) [\pth{\calP_{\mathrm{long}} \, \varphi}(s)] }_k = \pth{ \Us(s,v) [\pth{\calP_{\mathrm{long}} \, \varphi \1_{\llbracket (2c-1)N,(1-2c)N \rrbracket} } (s)]  }_k + N^{-D}, $$
where $ (\varphi \1_{\llbracket (2c-1)N,(1-2c)N \rrbracket})_j = \varphi_j \1_{\llbracket (2c-1)N,(1-2c)N \rrbracket}(j) $.
Moreover, using the property that $ \Us $ is an $ L^\infty $ contraction, we have
\begin{equation}\label{eq:Tshort_Contraction}
\abs{ \pth{ \pth{\sT(u,v) - \Us(u,v)} \varphi(u) }_k } \leq |u-v| \sup_{|j| \leq (1-2c)N,u<s<v} \abs{ \pth{\calP_{\mathrm{long}} \, \varphi}_j(s) } + N^{-D}.
\end{equation}
For $ |i| \leq (1-c)N $, assuming $ \calH_\alpha $ and on the event that Lemma \ref{l.ApproxUbar} holds, there exists a constant $ C>0 $ so that
$$ \abs{ \varphi_i(s) - \varphi_j(s) } \leq \abs{ \varphi_i(s) - \uu_i(s) }+\abs{ \varphi_j(s) - \uu_j(s) } + \abs{\uu_i(s) - \uu_j(s)} \lesssim N^{C\eps} \pth{ \frac{(Nt)^\alpha}{N^2 t} + \frac{|i-j|}{N^2 t} }. $$
For $ |i|>(1-c)N $, Lemma \ref{lem:phi_decay_rough} yields
$$ |\varphi_i(s) - \varphi_j(s)| \leq |\varphi_i(s)| + |\varphi_j(s)| \lesssim N^{-\frac{2}{3}+C\eps} (N+1-|i|)^{-\frac{1}{3}}.  $$
Therefore, using rigidity of singular values, we have
\begin{align*}
\pth{\calP_{\mathrm{long}}\varphi}_j(s) &= \sum_{|i-j|>l} \frac{\varphi_i(s) - \varphi_j(s)}{2N (s_i(s) - s_j(s))^2}\\
&\lesssim N^{1+C\eps} \sum_{|i-j|>l} \frac{1}{(i-j)^2} \pth{\frac{(Nt)^\alpha}{N^2 t} + \frac{|i-j|}{N^2 t}} + N^{-1+C\eps} \sum_{|i|>(1-c)N} N^{-\frac{2}{3}} (N+1-|i|)^{-\frac{1}{3}}  \\
&\lesssim N^{C\eps} \pth{ \frac{N}{l} \frac{(Nt)^\alpha}{N^2 t} + \frac{1}{Nt} }.
\end{align*}
Combined with \eqref{eq:Tshort_Contraction}, this implies the desired claim with $ |k| \leq (1-3c)N $. Note that $ c>0 $ is arbitrary, and thus it concludes the proof.
\end{proof}

Further, we will show that we have nice control for a regularization of the short-range dynamics with a well-behaved initial data. To do this, we follow the techniques developed in \cite{bourgade2017eigenvector}. Consider some fixed times $ u<t $, the short-range parameter $ l $, and define an averaging space window scale $ r $. Throughout the remaining parts of this section, for any fixed arbitrarily small $ \eps>0 $, we make the following assumption on these parameters
\begin{equation}\label{e.ParameterAssumption}
N^{30\eps}(t-u) < N^{20\eps} \dfrac{l}{N} < N^{10\eps} r <t.
\end{equation}
For a fixed index $ k $, as in \cite{bourgade2017eigenvector}, we define the flattening operator with parameter $ a>0 $ by
\begin{equation*}
\left(\flat_a f \right)_j(v) := 
\left\{
\begin{aligned}
& f_j(v) & \ & \mbox{if}\ |j-k|\leq a\\
& \uu_k(t) & \ & \mbox{if}\ |j-k| >a
\end{aligned}
\right.\ \ \ \mbox{for }u \leq v \leq t,
\end{equation*}
and the averaging operator
$$ \left( \av f \right)_j := \dfrac{1}{|\llbracket Nr,2Nr  \rrbracket|} \sum_{a \in \llbracket Nr,2Nr  \rrbracket} \left( \flat_a f \right)_j $$
As shown in \cite[Equation (7.4)]{bourgade2017eigenvector}, the averaging operator can also be represented as a combination of Lipschitz function, i.e. there exists a Lipschitz function $ h $ with $ |h_i-h_j| \leq \tfrac{|i-j|}{Nr} $ such that
\begin{equation}\label{e.AveragingLipschitz}
\left( \av f \right)_j = h_j f_j + (1-h_j) \uu_k(t).
\end{equation}

Finally, for $ u<v<t $, consider the regularized dynamics
\begin{equation*}
\left\{
\begin{aligned}
& \dfrac{\d}{\d v}\m_j(v) = \left( \calP_{\mathrm{short}}(v)\m \right)_j,\\
& \m(u)=\av \varphi(u)
\end{aligned}
\right.
\end{equation*}

The following lemma shows that the averaging the regularized dynamics gives good approximation for $ \uu_k $.

\begin{lemma}\label{l.ModificationDynamics}
Assume $ \calH_{\alpha} $. Let $ c>0 $ be any fixed small constant. There exists a constant $ C>0 $ such that for any $ \eps>0 $, $ N^{-1+C\eps}<\eta,t<1 $, $ \tfrac{t}{2}\leq u<v \leq t $, $ l>N^\eps $, $ j,k \in \llbracket (2c-1)N,(1-2c)N \rrbracket $ such that $ |\gamma_j - \gamma_k|<10r $, and $ z=\gamma_j + \i\eta $, we have
\begin{multline}\label{e.ModificationDynamics}
\left| \dfrac{1}{2N} \im \sum_{|i-j|<l} \dfrac{\m_i(v)}{s_i(v) - z} - \left( \dfrac{1}{2N} \im \sum_{|i-j|<l} \dfrac{1}{s_i(v) - z} \right) \uu_k(v) \right|\\
\prec \left( \dfrac{r}{Nt} + \dfrac{\eta}{Nt} + \dfrac{(Nt)^\alpha}{N^2 t} \left( \dfrac{l}{Nr} + \dfrac{N \eta}{l} + \dfrac{N(t-u)}{l} + \dfrac{1}{N\eta} \right) \right).
\end{multline}
\end{lemma}
\begin{proof}
We decompose the upper line of \eqref{e.ModificationDynamics} into three parts
$$ \dfrac{1}{2N} \im \sum_{|i-j|<l} \dfrac{\m_i(v)}{s_i(v) - z} - \left( \dfrac{1}{2N} \im \sum_{|i-j|<l} \dfrac{1}{s_i(v) - z} \right) \uu_k(v) =:I_1+I_2+I_3, $$
where
\begin{align*}
I_1 &=\dfrac{1}{2N} \im \sum_{|i-j|<l} \dfrac{\left( \Us(u,v) \av \varphi(u) - \av \Us(u,v) \varphi(u) \right)_i}{s_i(v) - z} \\
I_2 &=\dfrac{1}{2N} \im \sum_{|i-j|<l} \dfrac{\left( \av \Us(u,v) \varphi(u) - \av \sT(u,v) \varphi(u) \right)_i}{s_i(v) - z} \\
I_3 &=\dfrac{1}{2N} \im \sum_{|i-j|<l} \dfrac{\left( \av \sT(u,v) \varphi(u) \right)_i - \uu_k(v)}{s_i(v) - z}.
\end{align*}

For the first term $ I_1 $, Note that
\begin{multline*}
\left( \Us(u,v) \av \varphi(u) - \av \Us(u,v) \varphi(u) \right)_i =\\
\dfrac{1}{|\llbracket Nr,2Nr  \rrbracket|} \sum_{a \in \llbracket Nr,2Nr  \rrbracket}  \left( \Us(u,v) \flat_a \varphi(u) - \flat_a \Us(u,v) \varphi(u) \right)_i.
\end{multline*}
When $ |i-k|< a- N^\eps l $, we have
$$ \left( \flat_a \Us(u,v) \varphi(u) \right)_i = \left( \Us(u,v) \varphi(u)  \right)_i, $$
and the finite speed of propagation \eqref{e.FiniteSpeed} yields
$$ \left( \Us(u,v) \flat_a \varphi(u) \right)_i= \left( \Us(u,v) \varphi(u)  \right)_i + N^{-D}. $$
This gives
$$ \left( \Us(u,v) \flat_a \varphi(u) - \flat_a \Us(u,v) \varphi(u) \right)_i < N^{-D} $$
Similarly, this bound also holds in the case $ |i-k|>a+N^\eps l $. Now suppose $ a-N^\eps l \leq |i-k| \leq a+N^\eps l $. Applying \eqref{e.GapUbar} and Hypothesis $ \calH_{\alpha} $, we obtain
\begin{align*}
&\quad \left| \left( \Us(u,v) \flat_a \varphi(u) - \flat_a \Us(u,v) \varphi(u) \right)_i \right|\\
&\leq \max_{m:||m-k|-a| \leq 2 N^\eps l} |\varphi_m(v) - \uu_k(t)|+N^{-D}\\
&\leq \max_{m:||m-k|-a| \leq 2 N^\eps l} |\varphi_m(v) - \uu_m(v)|+ \max_{m:||m-k|-a| \leq 2 N^\eps l} |\uu_m(v) - \uu_k(t)| +N^{-D}\\
&\leq N^{C\eps} \left( \dfrac{(Nt)^\alpha}{N^2 t} + \dfrac{r+(t-u)}{Nt} \right).
\end{align*}
Combined with the estimate above, this implies
\begin{equation}\label{e.LemmaRegularization_I1}
I_1 \leq N^{C\eps} \dfrac{l}{Nr} \left( \dfrac{(Nt)^\alpha}{N^2 t} + \dfrac{r+(t-u)}{Nt} \right).
\end{equation}

For the term $ I_2 $, note that $ |i-j|<l $ implies $ |i| \leq (1-c)N $. Therefore, using the Lipschitz representation of the averaging operator \eqref{e.AveragingLipschitz}, the short-range approximation \eqref{e.ShortRange} gives us
\begin{multline*}
\left| \left( \av \Us(u,v) \varphi(u) - \av \sT(u,v) \varphi(u) \right)_i \right|\\
\leq |\left( \Us(u,v)\varphi(u) - \sT(u,v)\varphi(u) \right)_i| \leq N^{C\eps} (t-u) \left( \dfrac{N}{l}\dfrac{(Nt)^\alpha}{N^2 t} + \dfrac{1}{Nt} \right).
\end{multline*}
This shows
\begin{equation}\label{e.LemmaRegularization_I2}
I_2 \leq N^{C\eps} (t-u) \left( \dfrac{N}{l}\dfrac{(Nt)^\alpha}{N^2 t} + \dfrac{1}{Nt} \right).
\end{equation}

Finally, for the term $ I_3 $, by the Lipschitz representation of the averaging opeator \eqref{e.AveragingLipschitz}, it can be rewritten in the following way,
\begin{align*}
I_3 &= \dfrac{1}{2N} \im\sum_{-N \leq i \leq N} \dfrac{h_j(\varphi_i(v) - \uu_k(v))}{s_i - z} - \dfrac{1}{2N} \im \sum_{|i-j| \geq l} \dfrac{h_j(\varphi_i(v) - \uu_k(v))}{s_i -z}  \\
&\quad +\dfrac{1}{2N} \im \sum_{|i-j|< l}\dfrac{(h_i - h_j)(\varphi_i(v) - \uu_k(v))}{s_i -z} + \dfrac{1}{2N} \im \sum_{|i-j|< l}\dfrac{(1-h_i)(\uu_k(t) - \uu_k(v))}{s_i -z}\\
&=:J_1+J_2+J_3+J_4.
\end{align*}
Using \eqref{e.ApproxUbar} and \eqref{e.ImprovedLocalAverage}, we control $ J_1 $ in the following way
\begin{align*}
J_1 &\leq \dfrac{e^{\frac{v}{2}}}{2N} \left( \im \frakS_v(z) - e^{-\frac{v}{2}}\dfrac{\im S_v(z)}{\im S_0(z_v)}\im \frakS_0(z_v) \right) + \im S_v(z) \left( \dfrac{\im \frakS_0(z_v)}{N \im S_0(z_v)} - \uu_k(v) \right)\\
&\leq N^{C\eps} \left( \dfrac{(Nt)^\alpha}{N^3 t \eta} + \dfrac{1}{N^2 t} + \dfrac{\eta +r}{Nt} \right).
\end{align*}
Applying \eqref{e.GapUbar} to estimate $ J_2 $, we obtain
\begin{align*}
J_2 &\leq \dfrac{1}{2N} \sum_{|i-j| \geq l} \dfrac{\eta}{(\gamma_i - \gamma_j)^2 + \eta^2}\left( |\varphi_i(v) - \uu_i(v)| + |\uu_i(v) - \uu_k(v)| \right)\\
&\leq \dfrac{N^{C\eps}}{2N} \sum_{|i-j| \geq l} \dfrac{\eta}{(\gamma_i - \gamma_j)^2 + \eta^2} \left( \dfrac{(Nt)^\alpha}{N^2 t} + \dfrac{|i-j|}{N^2 t} + \dfrac{Nr}{N^2 t} \right)\\
&\leq N^{C\eps} \left( \dfrac{(Nt)^\alpha}{N^2 t}\dfrac{N \eta }{l}+\dfrac{\eta}{Nt}+\dfrac{r}{Nt} \right).
\end{align*}
Similarly, by the Lipschitz property of $ \{h_i\} $, we estimate $ J_3 $ as follows
\begin{multline*}
J_3 \leq \dfrac{1}{2N} \sum_{|i-j| < l} \dfrac{\eta}{(\gamma_i - \gamma_j)^2 + \eta^2}\dfrac{|i-j|}{Nr}\left( \dfrac{(Nt)^\alpha}{N^2 t} + \dfrac{|i-j|}{N^2 t} + \dfrac{Nr}{N^2 t} \right)\\ \leq N^{C\eps} \dfrac{l}{Nr} \left( \dfrac{(Nt)^\alpha}{N^2 t} + \dfrac{r}{Nt} \right) \leq N^{C\eps} \left( \dfrac{(Nt)^\alpha}{N^2 t} \dfrac{l}{Nr} + \dfrac{r}{Nt} \right).
\end{multline*}
Using the same arguments, by \eqref{e.GapUbar}, we have
$$ J_4 \leq \dfrac{N^{C\eps}}{2N} \sum_{|i-j| < l} \dfrac{\eta}{(\gamma_i - \gamma_j)^2 + \eta^2} \dfrac{t-v}{Nt} \leq N^{C\eps} \dfrac{r}{Nt}. $$
Together with the previous estimates, this leads to
$$ I_3 \leq N^{C\eps} \left( \dfrac{r}{Nt} + \dfrac{\eta}{Nt} + \dfrac{(Nt)^\alpha}{N^2 t} \left( \dfrac{l}{Nr} + \dfrac{N \eta}{l} + \dfrac{1}{N\eta} \right) \right) $$
Combined with \eqref{e.LemmaRegularization_I1} and \eqref{e.LemmaRegularization_I2}, we obtain the desired result.
\end{proof}

Finally, we have all the tools to prove Lemma \ref{l.Bootstrap}.

\begin{proof}[Proof of Lemma \ref{l.Bootstrap}]
We fix some small $ c>0 $ and consider an arbitrarily small $ \eps>0 $. Throughout the whole proof, we do all estimates on the overwhelming probability event where Lemma \ref{l.FiniteSpeed}, Lemma \ref{l.ShortRange} and Lemma \ref{l.ModificationDynamics} hold. For a fixed index $ k \in \llbracket (2c-1)N,(1-2c)N \rrbracket $, we have
\begin{equation*}
|\varphi_k(t) - \m_k(t)| \leq
\left| \left( (\sT(u,t) -\Us(u,t)) \varphi(u) \right)_k \right| + \left| \left( \Us(u,t)(\varphi(u) -\m(u)) \right)_k  \right|.
\end{equation*}
By the definition of the averaging opeartor, we know that $ \m(u) = \av \varphi(u)=\varphi(u) $ on the set $ \{j:|j-k|\leq Nr\} $. Therefore, combined with the finite speed of propagation estimate \eqref{e.FiniteSpeed} for the second term and the short-range approximation \eqref{e.ShortRange} for the first term, we obtain
\begin{equation}\label{e.UkMk}
|\varphi_k(t) - \m_k(t)| \leq  N^{C\eps} (t-u)\left( \dfrac{(Nt)^\alpha}{N^2 t}\dfrac{N}{l} + \dfrac{1}{Nt} \right) + N^{-2022}.
\end{equation}

It suffices to estimate $ |\m_k(t) - \uu_k(t)| $. Consider the function
$$ M(v) := \max_{-N \leq i \leq N} \left( \m_i(v) - \uu_k(t) \right).  $$
Similarly as in Lemma \ref{l.MaxPrin}, we can show a parabolic maximum principle for $ M $ and consequently $ M $ decreases in time. Moreover, note that $ \m_i(u)=\uu_k(t) $ if $ |i-k|\geq 2Nr $.

Let $ j=j(v) $ to denote the index that attains the maximum. If there exists a time $ u \leq v \leq t $ such that $ |j-k|>3Nr $, then in this case the finite speed propagation \eqref{e.FiniteSpeed} gives us
\begin{equation}\label{e.M_tSpecial}
M(t) \leq M(v) = \m_j(v) - \uu_k(t) \leq N^{-2022}.
\end{equation}
On the other hand, now we assume that $ |j(v) - k|<3Nr $ for all $ u \leq v \leq t $. In this case, we have
$$ \dfrac{\d}{\d v}\left( \m_j(v) - \uu_k(t) \right) = \sum_{|i-j| < l} \dfrac{\m_i(v) - \m_j(v)}{2N(s_i(v) - s_j(v))^2} \leq \dfrac{1}{2N} \sum_{|i-j| < l} \dfrac{\m_i(v) - \m_j(v)}{(s_i(v) - s_j(v))^2+\eta^2}. $$
This gives us
\begin{equation*}
\dfrac{\d}{\d v}\left( \m_j(v) - \uu_k(t) \right)
\leq \dfrac{1}{2N\eta} \im \sum_{|i-j|<l} \dfrac{\m_i(v)}{s_i(v) - z} - \left( \dfrac{1}{2N\eta} \im \sum_{|i-j|<l} \dfrac{1}{s_i(v) - z} \right) \m_j(v)
\end{equation*}
and therefore
\begin{align*}
\dfrac{\d}{\d v}\left( \m_j(v) - \uu_k(t) \right)
&\leq \dfrac{1}{\eta}\left( \dfrac{1}{2N} \im \sum_{|i-j|<l} \dfrac{\m_i(v)}{s_i(v) - z} - \left( \dfrac{1}{2N} \im \sum_{|i-j|<l} \dfrac{1}{s_i(v) - z} \right) \uu_k(v) \right)\\
&\quad + \dfrac{1}{\eta} \left(\dfrac{1}{2N} \im \sum_{|i-j|<l} \dfrac{1}{s_i(v) - z}\right) \left( \uu_k(v) - \uu_k(t) \right)\\
&\quad + \dfrac{1}{\eta} \left(\dfrac{1}{2N} \im \sum_{|i-j|<l} \dfrac{1}{s_i(v) - z}\right) \left( \uu_k(t) - \m_j(v) \right).
\end{align*}
Applying Lemma \ref{l.ModificationDynamics} and Lemma \ref{l.ApproxUbar} yields
\begin{equation*}
\dfrac{\d}{\d v}M(v^+) \lesssim
-\dfrac{1}{\eta}M(v) + \dfrac{N^{C\eps}}{\eta}\left( \dfrac{r}{Nt} + \dfrac{\eta}{Nt} + \dfrac{(Nt)^\alpha}{N^2 t} \left( \dfrac{l}{Nr} + \dfrac{N \eta}{l} + \dfrac{N(t-u)}{l} + \dfrac{1}{N\eta} \right) \right),
\end{equation*}
where the left-hand side represents the right derivative of $ M $ at time $ v $. Let $ \eta=\tfrac{(t-u)}{N^{\eps}} $, then above inequality leads to
$$ M(t) \leq N^{C\eps} \left( \dfrac{r}{Nt} + \dfrac{(Nt)^\alpha}{N^2 t} \left( \dfrac{l}{Nr} + \dfrac{N (t-u)}{l} + \dfrac{1}{N(t-u)} \right) \right) $$
Choosing
\begin{equation}\label{e.OptimalParameter}
r=\dfrac{(Nt)^{\frac{3\alpha}{4}}}{N},\ \ l=(Nt)^{\frac{\alpha}{2}},\ \ (t-u)=\dfrac{(Nt)^\frac{\alpha}{4}}{N},
\end{equation}
then we have
$$ M(t) < N^{C\eps} \dfrac{(Nt)^{\frac{3\alpha}{4}}}{N^2 t} $$
Similarly, this bound also holds for $ -\max_{-N \leq i \leq N} \left( \m_i(s) - \uu_k(t) \right) $. Combined with \eqref{e.UkMk} and \eqref{e.M_tSpecial}, this completes the proof.
\end{proof}

\section{Proofs for Quantitative Universality}\label{app:Comparison}
In this section, we prove the quantitative resolvent comparison Proposition \ref{prop:Resolvent_Comparison}.

The key idea is based on the Lindeberg exchange method (for a detailed introduction we refer to the monograph \cite{erdos2017book,van2014probability}). We first fix an ordering map of the indices $ \phi: \{(i,j):1 \leq i,j \leq N\} \to \llbracket N^2 \rrbracket $. For $ 0 \leq k \leq N^2 $, let $ H_k $ be the random matrix defined as
\begin{equation*}
(H_k)_{ij}=
\begin{cases}
X_{ij} & \textrm{if } \phi(i,j) \leq k\\
Y_{ij} & \textrm{otherwise}
\end{cases}
,
\end{equation*}
so that $ H_0=Y $ and $ H_{N^2}=X $. By telescoping summation, it suffices to show the following is true uniformly in $ 1 \leq k \leq N^2 $,
\begin{equation}\label{eq:Telescope}
\left| \E[F(\Tr f(H_k))] - \E[F(\Tr f(H_{k-1}))] \right| \leq  \frac{N^{C\eps}}{N^2} \pth{\frac{1}{\rho N^2} + \frac{\pth{\rho N^a}^5}{\sqrt{N}} + t \rho N^a}
\end{equation}

To prove \eqref{eq:Telescope}, we use the Helffer-Sj\"{o}strand formula. Let $ \chi $ be a smooth symmetric cutoff function such that $ \chi(y)=1 $ if $ |y|<N^{-a} $ and $ \chi(y)=0 $ if $ |y|>2 N^{-a} $, with $ \|\chi'\|_\infty \leq N^a $.
For any matrix $ H \in \sth{H_k}_{k=0}^{N^2} $, let $ \tH $ denote its Girko symmetrization
\begin{equation*}
\tH=\left(
\begin{matrix}
0 & H^\top\\
H & 0
\end{matrix}
\right)
\end{equation*}
Recall that the symmetrized singular values $ \sth{\sigma_i(H)}_{i=-N}^N $ are the eigenvalues of $ \tH $. With the cutoff function $ \chi $, applying Lemma \ref{lem:HS_Formula} to $ \tH $ yields
\begin{equation}\label{eq:HS}
\Tr f(H) = \int_\C g(z) \Tr (\tH-z)^{-1} \, \d^2 z,
\end{equation}
where $ \d^2 z $ is the Lebesgue measure on $ \C $ and
$$ g(z) := \frac{1}{\pi} \qth{ \i y f''(x) \chi(y) + \i \pth{ f(x) + \i y f'(x) } \chi'(y) },\ \ z=x+y\i. $$
The analysis of the comparison can be proceeded in the following steps:

\smallskip

\emph{Step 1: Approximation of $ \Tr f(H) $.}
We first truncate the integral in \eqref{eq:HS} and define
$$ \calT(H) := \int_{|y|>N^{-2}} g(z) \Tr (\tH-z)^{-1} \d^2 z. $$
The approximation error can be bounded by
\begin{align*}
\abs{\Tr f(H) - \calT(H)} &\lesssim \iint_{|y|<N^{-2},E<|x|<E+\rho} |f''(x)| \sum_{-N \leq k \leq N} \frac{y^2}{|\sigma_k - (x+y\i)|^2} \d x \d y\\
& \lesssim \int_{E<|x|<E+\rho} \frac{1}{\rho^2 N^2} \pth{\frac{1}{N^2} \sum_{-N \leq k \leq N} \frac{1}{|\sigma_k - (x+\frac{\i}{N})|^2} } \d x.
\end{align*}
For singular values near the origin, i.e. $ |k| \leq N^{C\eps} $, we have 
$$ \int_{E<|x|<E+\rho} \frac{1}{|\sigma_k - (x+\frac{i}{N})|^2} \d x \leq \int_{E<|x|<E+\rho} N^2 \d x \lesssim \rho N^2. $$
On the other hand, for $ |k| > N^{C\eps} $, by the rigidity of singular values, we have the following overwhelming probability bound
$$ \int_{E<|x|<E+\rho} \frac{1}{|\sigma_k - (x+\frac{i}{N})|^2} \d x \leq \frac{\rho}{|E-\gamma_k|^2}. $$
Combining the above two bounds together, we obtain
\begin{equation*}
\abs{\Tr f(H) - \calT(H)} \leq \frac{N^{C\eps}}{\rho N^2} + \frac{1}{\rho N^4} \sum_{|k|>N^{C\eps}} \frac{1}{|E-\gamma_k|^2} \lesssim \frac{N^{C\eps}}{\rho N^2} + \frac{1}{\rho N^3} \pth{ \frac{1}{N} \sum_{k \geq 1} \frac{1}{(k/N)^2} } \lesssim \frac{N^{C\eps}}{\rho N^2}
\end{equation*}
with overwhelming probability.

\smallskip

\emph{Step 2: Expansions and moment matching.} With the approximation by $ \calT(H) $, it suffices to show that
\begin{equation}\label{eq:Telescoping_Claim}
\abs{\E[F(\calT(H_k))] - \E[F(\calT(H_{k-1}))]} \lesssim \frac{N^{C\eps}}{N^2} \pth{ \frac{(\rho N^a)^5}{\sqrt{N}} + t \rho N^a }
\end{equation}
uniformly for all $ 1 \leq k \leq N^2 $. Now consider a fixed $ 1 \leq \omega \leq N^2 $ corresponding to the index $ (i,j) $, i.e. $ \phi(i,j)=\omega $. We rewrite the matrices $ H_\omega $ and $ H_{\omega-1} $ in the following way
$$ H_{\omega} = W + \frac{1}{\sqrt{N}}U  ,\ \ H_{\omega-1} = W + \frac{1}{\sqrt{N}}V, $$
where the matrix $ W $ coincides with $ H_{\omega-1} $ and $ H_{\omega} $ except on the $ (i,j) $ entry with $ W_{ij}=0 $. Then note that the matrices $ U,V $ satisfy $ U_{ij}=\sqrt{N}X_{ij} $ and $ V_{ij}=\sqrt{N}Y_{ij} $ and all other entries are zero. Recall the notation $ \tH $ for the Girko symmetrization. Consider the resolvents of the matrices $ \tW $ and $ \tH_\omega $
$$ R:=(\tW-z)^{-1},\ \ S:=(\tH_\omega-z)^{-1}. $$
The Taylor expansion yields
\begin{equation}\label{eq:Taylor}
\begin{aligned}
&\E[F(\calT(H_{\omega}))] - \E[F(\calT(H_{\omega-1}))]\\
&\quad = \sum_{k=1}^4 \E \qth{ \frac{F^{(k)}(\calT(W))}{k!} \pth{ (\calT(H_\omega) - \calT(W))^k - (\calT(H_{\omega-1}) - \calT(W))^k } }\\
&\quad \quad + O(\|F^{(5)}\|_\infty) \, \E \qth{ (\calT(H_\omega) - \calT(W))^5 + (\calT(H_{\omega-1}) - \calT(W))^5 }.
\end{aligned}
\end{equation}
We first control the term corresponding to the fifth derivative. By Lemma \ref{lem:Resolvent_Expansion}, the first order resolvent expansion gives us
$$ \frac{1}{N} \Tr S - \frac{1}{N} \Tr R = \frac{1}{\sqrt{N}} \Tr S \tU R. $$
Consequently,
$$ \abs{\calT(H_\omega) - \calT(W)} \leq \frac{1}{\sqrt{N}} \int_{|y|>N^{-2}} |g(z)| \abs{\Tr S(z) \tU R(z)} \d^2 z. $$
We can restrict the integral on the domain $ \sth{z=x+y \i : N^{-2}<|y|<2N^{-a},E<|x|<E+\rho} $ as the contribution outside this region is negligible. Moreover, a key observation is that the matrix $ \tU $ only has two non-zero entries. Thus,
\begin{multline*}
\abs{\calT(H_\omega) - \calT(W)} \lesssim  N^{\frac{1}{2}+C\eps} \int_{N^{-2}<|y|<2N^{-a},E<|x|<E+\rho} |g(z)| \bigg(\max_{k \neq \ell}|S_{k \ell}(z)| \bigg) \bigg(\max_{k \neq \ell} |R_{k \ell}(z)| \bigg) \d^2 z\\
+  N^{-\frac{1}{2}+C\eps} \int_{N^{-2}<|y|<2N^{-a},E<|x|<E+\rho} |g(z)| \pth{\max_k |S_{kk}(z)|} \pth{\max_{k}|R_{kk}(z)|} \d^2 z.
\end{multline*}
Note that in this integral domain, the scale of $ |y| $ is smaller than the natural size of the local law. Therefore, we will use a suboptimal version of the local semicircle law for a larger spectral domain, which was discussed in \cite{erdHos2013spectral,lee2018local}. For $ z=x+\i y $ in this integral domain, with overwhelming probability we have
$$ \max_{k,\ell} \abs{S_{k \ell}(z) - \delta_{k \ell} \, m_{sc}(z)} \leq N^{C\eps} \pth{\frac{1}{\sqrt{N}} + \Psi(z) },\ \ \Psi(z) = \frac{1}{Ny} + \sqrt{\frac{\im m_{sc}(z)}{Ny}}. $$
The same result also holds for $ R_{k \ell}(z) $. By Lemma \ref{lem:m_sc}, we have $ \Psi(z) \lesssim \frac{1}{\sqrt{Ny}} $ for $ z $ in the integral domain. Note that the contribution of the diagonal resolvent entries is negligible. Therefore, with overwhelming probability we have
$$ \abs{\calT(H_\omega) - \calT(W)} \lesssim N^{C\eps} N^{1/2} \int_{N^{-2}<|y|<2N^{-a},E<|x|<E+\rho} \frac{|g(z)|}{Ny} \d^2 z \lesssim N^{C\eps} N^{-1/2} \rho N^a. $$
Similarly, this bound also holds for $ |\calT(H_{\omega-1}) - \calT(W)| $, and we obtain 
$$ \E[(\calT(H_\omega) - \calT(W))^5] \lesssim \frac{N^{C\eps}}{N^{2}} \pth{ N^{-\frac{1}{2}} (\rho N^a)^5 },\ \ \E[(\calT(H_{\omega-1}) - \calT(W))^5] \lesssim \frac{N^{C\eps}}{N^{2}} \pth{ N^{-\frac{1}{2}} (\rho N^a)^5 }. $$ 
Hence the fifth order term in \eqref{eq:Taylor} is bounded by
\begin{equation}\label{eq:Taylor_Fifth}
O(\|F^{(5)}\|_\infty) \, \E \qth{ (\calT(H_\omega) - \calT(W))^5 + (\calT(H_{\omega-1}) - \calT(W))^5 } \lesssim \frac{N^{C\eps}}{N^{2}}  \frac{ (\rho N^a)^5 }{\sqrt{N}} .
\end{equation}

Now we consider the first term $ k=1 $ in the Taylor expansion \eqref{eq:Taylor}.
Denote
$$ \hat{R}:=\frac{1}{N} \Tr R,\ \ \hat{R}_X^{(m)} = \frac{(-1)^m}{N} \Tr (R\tU)^m R,\ \ \Omega_X := -\frac{1}{N} \Tr (R\tU)^5 S, $$
and also define
$$ \hat{R}_Y^{(m)}:=\frac{(-1)^m}{N} \Tr (R\tV)^m R,\ \ \Omega_Y := -\frac{1}{N} \Tr (R \tV)^5 (\tH_{\omega-1}-z)^{-1}. $$
Using the resolvent expansion (Lemma \ref{lem:Resolvent_Expansion}) up to the fifth order, we obtain
$$ \frac{1}{N} \Tr S = \hat{R} + \sum_{m=1}^4 N^{-\frac{m}{2}} \hat{R}_X^{(m)} + N^{-\frac{5}{2}} \Omega_X. $$
A Similar expansion also holds for $ (\tH_{\omega-1}-z)^{-1} $. Then we have
\begin{multline}\label{eq:Moment_Matched}
\E \qth{ F'(\calT(W)) \pth{ \calT(H_\omega) - \calT(H_{\omega-1}) } }\\
= \E \qth{ F'(\calT(W)) \int_{|y|>N^{-2}} g(z) \pth{ \sum_{m=1}^4 N^{-\frac{m}{2}+1} (\hat{R}_X^{(m)} - \hat{R}_Y^{(m)}) + N^{-\frac{3}{2}} (\Omega_X - \Omega_Y) } \d^2 z }
\end{multline}
A key observation is that for $ 1 \leq m \leq 3 $, the terms $ \hat{R}_X^{(m)} $ and $ \hat{R}_Y^{(m)} $ only depend on the first three moments of $ X_{ij} $ and $ Y_{ij} $. Recall that the first three moments of $ X_{ij} $ and $ Y_{ij} $ are identical. Therefore, the terms corresponding to $ 1 \leq m \leq 3 $ in \eqref{eq:Moment_Matched} makes no contribution.

\smallskip

\emph{Step 3: Higher order error.}
For the $ m=4 $ term in \eqref{eq:Moment_Matched}, note that
\begin{equation}\label{eq:Trace_Expansion}
\Tr (R \tU)^4 R = \sum_{1 \leq \ell \leq 2N} \sum_{\{\alpha_k,\beta_k\} = \{i+N,j\}} R_{\ell \alpha_1} \tU_{\alpha_1 \beta_1} R_{\beta_1 \alpha_2} \tU_{\alpha_2 \beta_2} R_{\beta_2 \alpha_3} \tU_{\alpha_3 \beta_3} R_{\beta_3 \alpha_4} \tU_{\alpha_4 \beta_4} R_{\beta_4 \ell}.
\end{equation}
A similar formula is also true for $ \Tr (R \tV)^4 R $. Note that typically we have $ \ell \neq \alpha_1 $ and $ \ell \neq \beta_4 $, but we may have $ \beta_1=\alpha_2 $, $ \beta_2=\alpha_3 $, $ \beta_3=\alpha_4 $. Moreover, the terms with either $ \ell=1+N $ or $ \ell=j $ are combinatorially negligible in the summation and therefore we can ignore these terms in the following computations. Recall that the difference between the fourth moments of $ X_{ij} $ and $ Y_{ij} $ is bounded by $ t $. Thus, we have
$$ \E \qth{ N (\hat{R}_X^{(4)} - \hat{R}_Y^{(4)}) } = \E \qth{ \Tr (R \tU)^4 R - \Tr (R \tV)^4 R } \lesssim N t \pth{\max_{k \neq \ell}|R_{k \ell}|}^2 \pth{\max_{k} |R_{kk}|}^3. $$
As mentioned above, for the integral in \eqref{eq:Moment_Matched} we can restrict the integral domain to $ N^{-2}<|y|<2N^{-a} $ and $ E<|x|<E+\rho $. In this region, the entries of the resolvent are bound by $ \max_{k \neq \ell}|R_{k \ell}(z)| \lesssim \frac{N^{C\eps}}{\sqrt{Ny}} $ and $ \max_{k}|R_{kk}(z)| \lesssim N^{C\eps} $. As a consequence,
\begin{multline}\label{eq:Moment_Fourth}
\left| \E \qth{ F'(\calT(W)) \int_{|y|>N^{-2}} g(z) N^{-\frac{4}{2}+1} (\hat{R}_X^{(4)} - \hat{R}_Y^{(4)}) \d^2 z } \right|\\
\lesssim N^{C\eps} \frac{t}{N} \int_{N^{-2}<|y|<2N^{-a},E<|x|<E+\rho} \frac{|g(z)|}{N y} \d^2 z \lesssim \frac{N^{C\eps}}{N^2} (t \rho N^a).
\end{multline}

For the term $ \Omega_X-\Omega_Y $, since these terms involve the higher moments of $ X_{ij} $ and $ Y_{ij} $, we simply bound it by the size of $ \Omega_X $ and $ \Omega_Y $. By a similar expansion as in \eqref{eq:Trace_Expansion} and the local law, we have $ |\Omega_X|,|\Omega_Y| \lesssim \frac{N^{C\eps}}{Ny} $. Therefore,
\begin{multline}\label{eq:Moment_Fifth}
\left| \E \qth{ F'(\calT(W)) \int_{|y|>N^{-2}} g(z) N^{-\frac{3}{2}} \Omega_X \d^2 z } \right|\\
\lesssim N^{C\eps} N^{-\frac{3}{2}} \int_{N^{-2}<|y|<2N^{-a},E<|x|<E+\rho} \frac{|g(z)|}{Ny} d^2 z \lesssim \frac{N^{C\eps}}{N^2} \frac{\rho N^a}{\sqrt{N}} \leq \frac{N^{C\eps}}{N^2} \frac{(\rho N^a)^5}{\sqrt{N}}.
\end{multline}
The same bound also holds for $ \Omega_Y $.

Finally, as explained in classical literature of random matrix theory (see e.g. \cite[Theorem 17.4]{erdos2017book}), the contributions of higher order terms in the Taylor expansion \eqref{eq:Taylor} are of smaller order. Consequently, combining \eqref{eq:Taylor_Fifth}, \eqref{eq:Moment_Fourth} and \eqref{eq:Moment_Fifth} yields the claim \eqref{eq:Telescoping_Claim}, which implies the desired result \eqref{eq:Resolvent_Comparison}.



\bibliography{Quant_Hard_Edge}
\bibliographystyle{alpha}

\end{document}